\documentclass[oneside,11pt]{amsart}

\makeatletter
\@namedef{subjclassname@2020}{
  \textup{2020} Mathematics Subject Classification}
\makeatother
\usepackage[a4paper,width=170mm,top=27mm,bottom=27mm]{geometry}
\usepackage{cases}
\usepackage{amsmath,amssymb}
\usepackage[hidelinks]{hyperref}
\usepackage{esint}

\newtheorem{theorem}{Theorem}[section]
\newtheorem{lemma}[theorem]{Lemma}

\newtheorem{corollary}[theorem]{Corollary}

\theoremstyle{definition}
\newtheorem{remx}[theorem]{Remark}
\newenvironment{remark}
  {\pushQED{\qed}\remx}
  {\popQED\endremx}

\newcommand{\la}{\langle}
\newcommand{\ra}{\rangle}

\newcommand{\N}{\mathbb{N}}
\newcommand{\R}{\mathbb{R}}
\newcommand{\C}{\mathbb{C}}
\newcommand{\T}{\mathbb{T}}

\newcommand{\ld}{\lambda}

\newcommand{\vare}{\varepsilon}

\newcommand{\pt}{\partial}
\newcommand{\bg}{\Big}

\newcommand{\Z}{{\mathbb{Z}}}

\numberwithin{equation}{section}

\begin{document}
\address{Yongming Luo
\newline \indent
Faculty of Computational Mathematics and Cybernetics
\newline \indent Shenzhen MSU-BIT University, China}
\email{luo.yongming@smbu.edu.cn}

\title[Almost sure well-posedness for periodic NLS]
{Almost sure local well-posedness for the nonlinear Schr\"odinger equations on $\T^d$ with non-algebraic nonlinearity}
\author{Yongming Luo}

\begin{abstract}
We study the Cauchy problem for the nonlinear Schr\"odinger equation
on $\mathbb T^d$ with random initial data and a general non-algebraic
power-type nonlinearity.  We establish almost sure local well-posedness
in every spatial dimension and for the whole mass-supercritical range
allowed by the natural condition $0<s_{\mathrm c}<1+a$.  The main new ingredient is a frequency-gaining probabilistic refinement
of the Galilean bilinear estimates recently developed by Kwak and Kwon \cite{KwakKwon}.
In the random setting, the gauge decomposition gives rise to three new
types of terms: a mean-free coefficient, an opposite-phase interaction,
and a scalar remainder.  We control them by new resonance counting and
large deviation arguments, and close the local theory through a
phase-adapted two-component contraction.  In the energy-critical case,
our result extends the low-dimensional algebraic theories of
Nahmod--Staffilani \cite{NahmodStaffilani15} and Yue \cite{Yue21} to every dimension $d\geq3$, including the
higher-dimensional non-algebraic models.
\end{abstract}

\keywords{Nonlinear Schr\"odinger equation, random initial data, non-algebraic nonlinearity, periodic dispersive equations, bilinear estimates, gauge transform}
\subjclass[2020]{35Q55, 35R60, 60H30}

\maketitle
\section{Introduction}
In this paper, we study the random initial value problem for the nonlinear
Schr\"odinger equation (NLS)
\begin{equation}\label{eq:nls}
 (i\partial_t+\Delta)u=\lambda F(u),
 \qquad
 F(z)=|z|^az,
 \qquad \lambda=\pm1,
\end{equation}
on the torus $\mathbb T^d=(\mathbb R/2\pi\mathbb Z)^d$.  We assume that
\begin{equation*}
 a>\frac4d,
 \qquad
 s_{\mathrm c}=\frac d2-\frac2a,
 \qquad
 0<s_{\mathrm c}<1+a.
\end{equation*}
Thus the equation is mass-supercritical.  The upper condition is a natural constraint
suggested by the regularity of the nonlinearity, since the map
$z\mapsto |z|^az$ is in general no smoother than $C^{1+a}$.

For $\delta>0$, we consider the Gaussian Fourier data
\begin{equation}\label{eq:random-data}
 \phi^\omega(x)
 =\sum_{n\in\mathbb Z^d}b_ng_n(\omega)e^{in\cdot x},
 \qquad
 b_n=\langle n\rangle^{-d/2-s_{\mathrm c}+\delta},
\end{equation}
where $\{g_n\}_{n\in\mathbb Z^d}$ is a family of independent standard
complex Gaussian random variables.  Equivalently, $\phi^\omega$ is the
randomization of
\[
 \phi(x)=\sum_{n\in\mathbb Z^d}b_ne^{in\cdot x}.
\]
A standard Gaussian-series argument shows that, almost surely,
$\phi^\omega\in H^s(\mathbb T^d)$ for every $s<s_{\mathrm c}-\delta$, while
$\phi^\omega\notin H^{s_{\mathrm c}-\delta}(\mathbb T^d)$.  In particular,
the data lie strictly below the deterministic critical regularity.  Our goal
is to construct a local solution for almost every realization of
\eqref{eq:random-data}, without assuming that the power $a$ is algebraic.

\subsection{Background and related works}
The NLS is one of the basic models for nonlinear dispersive waves.  It
appears, for instance, in nonlinear optics and in the mean-field description
of Bose--Einstein condensates; see \cite{Cazenave2003,KevrekidisEtAl2015} and
the references therein.  The periodic setting is natural for waves under
confinement or recurrent propagation.  From the mathematical point of view,
however, the torus behaves quite differently from $\mathbb R^d$: there is no
spatial decay at infinity, the spectrum is discrete, and the resonant
relations between lattice frequencies become an essential part of the
analysis.

For the equation \eqref{eq:nls}, the Euclidean scaling gives the critical
index $s_{\mathrm c}$.  Although scaling is not an exact symmetry on a fixed
torus, this number still indicates the expected threshold for local
well-posedness.  Below $H^{s_{\mathrm c}}$, the deterministic problem is in
general unstable (see e.g. \cite{ill_posed}).  Randomization provides a possible way around this
obstruction.  It does not improve Sobolev differentiability, but it improves
space-time integrability and the summability of Fourier interactions.  The
main question is whether these probabilistic gains are sufficient to recover
a meaningful flow below the deterministic threshold.

The probabilistic study of periodic NLS goes back to the fundamental works of
Bourgain \cite{BourgainProb1,BourgainProb2}.  In
\cite{BourgainProb1}, Bourgain combined finite-dimensional Hamiltonian
truncations, invariant Gibbs measures, and Fourier restriction spaces to
construct global dynamics for random data below the deterministic
regularity.  The argument introduced a scheme which has remained basic in
the subject: one first studies finite-dimensional equations, obtains bounds
which are uniform in the truncation parameter by using invariance of the
measure, and then passes to an infinite-dimensional limit.  In
\cite{BourgainProb2}, Bourgain considered the two-dimensional defocusing cubic NLS.
There the Gaussian free field is supported below $L^2(\mathbb T^2)$, and the
cubic term is not directly defined.  The equation has to be Wick ordered
before the limiting flow can be constructed.  These works showed, in
particular, that for very rough random data the correct nonlinear equation
may only become visible after a suitable approximation and renormalization.

This point of view has been developed much further in recent years.  Deng,
Nahmod and Yue \cite{DNY1} proved the invariance of Gibbs measures and the
existence of global strong solutions for two-dimensional defocusing NLS with
Wick-ordered power nonlinearities.  One of the main tools in their work is
the random averaging operator, which keeps track of the way randomness is
propagated through the nonlinear iteration.  In \cite{DNY2}, the same authors
developed the random tensor theory and obtained a general multilinear
framework for probabilistic dispersive equations.  These ideas have also
been used outside the NLS setting.  For example, Bringmann, Deng, Nahmod and
Yue \cite{BDNY24} combined random tensor estimates with a paracontrolled
expansion to construct the Gibbs dynamics for the three-dimensional cubic
nonlinear wave equation.

A different, though closely related, use of randomization was introduced by
Burq and Tzvetkov \cite{BurqTzvetkov1,BurqTzvetkov2}.  Instead of starting
from a distinguished Gibbs distribution, they randomized an arbitrary
deterministic datum in an eigenfunction basis on a compact manifold and
proved supercritical local and global results for nonlinear wave equations (NLW).
As in Bourgain's work, the solution is decomposed into a rough random linear
part and a smoother nonlinear correction.  The difference is that the gain
comes from the improved integrability of the randomized linear evolution,
and the construction is not tied to a particular invariant measure.  This
has become another standard approach in the probabilistic Cauchy theory nowadays.

The periodic works most closely related to the present paper are those of
Nahmod and Staffilani \cite{NahmodStaffilani15} and Yue \cite{Yue21}.  In \cite{NahmodStaffilani15}, Nahmod
and Staffilani studied the three-dimensional quintic NLS below the energy
space.  A key point in their paper is that Wick normalization alone does not
remove all the energy-critical resonant terms.  There remains a scalar
resonant piece which has to be eliminated by a gauge transform.  They
therefore solve a gauged equation and then return to the original NLS by
undoing the phase.  As we shall see, a gauge transform will also be forced
upon us, although it arises from a different counting obstruction.  Yue \cite{Yue21}
later extended this circle of ideas to the cubic NLS on $\mathbb T^d$,
$d\geq3$, using atomic and Fourier restriction spaces.  In particular, the
three-dimensional quintic and four-dimensional cubic energy-critical models
were covered in these two works.

It is worth pointing out that most random-data results on compact manifolds
are restricted to algebraic nonlinearities.  The reason is purely analytic. Indeed, on the Euclidean space $\mathbb R^d$, the dispersive decay of the linear flow gives
strong Strichartz and bilinear estimates.  On a compact manifold, the
available estimates are usually frequency localized and may lose
derivatives.  They are also closely tied to the arithmetic of the resonance
set or to the geometry of the eigenfunctions; see, for instance, the Fourier
restriction estimates of Bourgain \cite{Bourgain1,Bourgain2} and the
eigenfunction estimates of Burq, G\'erard and Tzvetkov
\cite{Burq1,Burq2,Burq3}.  For an algebraic nonlinearity, one may expand the
nonlinear term into finitely many multilinear frequency interactions and
then combine lattice counting with Wiener-chaos estimates.  A non-algebraic
power admits no such finite expansion.  Moreover, for small powers the
relevant derivatives are only H\"older continuous, and several standard
multilinear or contraction arguments are no longer available.

On the Euclidean space, a similar random initial value problem can be
formulated by means of Wiener randomization.  Thanks to the stronger
dispersion on $\mathbb R^d$, non-algebraic nonlinearities can still be
handled.  In particular, Oh, Okamoto, and Pocovnicu \cite{OhNonalg19} proved
probabilistic well-posedness results for the energy-critical NLS with
non-algebraic nonlinearities in dimensions five and six.  Roughly speaking,
they separate the nonlinear terms according to the position of the random
linear factor and use Euclidean bilinear estimates to recover the missing
derivative.  We shall see that in out case, the same strategy cannot be directly transferred to the torus,
since the corresponding bilinear gain is unavailable unless one makes a much
finer use of the resonance relation.

There is also a deterministic theory for periodic NLS with non-algebraic
nonlinearities.  To our knowledge, Lee \cite{LeeNonalg19} obtained the first
critical result of this type on $\mathbb T^3$, by combining a Bony
linearization with periodic bilinear estimates.  The argument requires enough
Lipschitz regularity of the nonlinear coefficients and therefore does not
cover the whole range in which these coefficients are merely H\"older
continuous.  This problem was recently solved by Kwak and Kwon
\cite{KwakKwon}.  Their idea is to keep the Bony coefficient as a whole,
rather than expanding it into many smaller pieces, and to design a bilinear
estimate adapted to that coefficient.  The Galilean transform, which is not
usually used in a local theory on the torus, plays a central role in their
argument.  Together with the space-time Besov space $Z^s$, this yields
critical local well-posedness in every dimension and for the whole
mass-supercritical range allowed by $s_{\mathrm c}<1+a$.  In the small-power
regime, where the flow map is not Lipschitz, their existence proof uses
approximation and compactness.

For some other recent developments related to rough periodic or stochastic
NLS, we refer to Colliander--Oh \cite{Oh_Co_Prob}, Oh--Sosoe--Tolomeo
\cite{Oh_Inv}, Gubinelli--Koch--Oh \cite{GKO24}, and Fan--Mendelson
\cite{Fan24}.

\subsection{Main result}
We now state the main result.  The nonlinear estimates naturally lead us to the
gauged NLS
\begin{equation}\label{eq:gauged}
 (i\partial_t+\Delta)y=\lambda G(y),
 \qquad y(0)=\phi^\omega,
\end{equation}
where
\begin{equation}\label{eq:ca-def}
 c_a=1+\frac a2,
 \qquad
 \mu(y)(t)=\fint_{\mathbb T^d}|y(t,x)|^a\,dx,
 \qquad
 G(y)=F(y)-c_a\mu(y)y.
\end{equation}
Since $\partial_zF(z)=c_a|z|^a$, the number $c_a\mu(y)$ is exactly the
spatial mean of the coefficient $\partial_zF(y)$.  By direct calculation, if
$y$ solves \eqref{eq:gauged}, then
\begin{equation}\label{eq:inverse-gauge}
 u(t,x)=\exp\left(-i\lambda c_a\int_0^t\mu(y)(s)\,ds\right)y(t,x)
\end{equation}
solves the original equation \eqref{eq:nls}.

\begin{theorem}\label{main_thm}
The gauged NLS \eqref{eq:gauged} is almost surely locally well-posed with
respect to the random initial data $\phi^\omega$.  More precisely, there
exist $C,c,\gamma,\delta_0,\vare_0>0$ such that for each
$0<\delta<\delta_0$ and $0<T\ll1$, there exists
$\Omega_T\subset\Omega$ with
\[
 \mathbb P(\Omega_T^c)
 \leq C\exp\left(-cT^{-\gamma}\right),
\]
such that for each $\omega\in\Omega_T$ there exists a unique solution $y$ to
\eqref{eq:gauged}, with $y|_{t=0}=\phi^\omega$, in the class
\[
 y=e^{it\Delta}\phi^\omega+w
 \in e^{it\Delta}\phi^\omega+Y^{s_{\mathrm c}+\vare_0}(I_T),
\]
where $I_T=[-T,T]$ and $Y^{s_{\mathrm c}+\vare_0}$ is defined in Section
\ref{sec_zs_space}.
\end{theorem}

Let us make a few comments on the statement.  First, the number
$\vare_0>0$ is fixed independently of $\delta\in(0,\delta_0)$.  Thus the
nonlinear part gains a uniform positive amount of regularity beyond the
scaling index.  Second, no algebraic assumption on $a$ is imposed: the result
holds in every dimension and for every mass-supercritical power satisfying
$0<s_{\mathrm c}<1+a$.  Finally, by undoing the gauge in
\eqref{eq:inverse-gauge}, we obtain the corresponding almost sure local theory
for \eqref{eq:nls}, with uniqueness understood in the associated
phase-adapted class.  In the energy-critical case $s_{\mathrm c}=1$, or
$a=4/(d-2)$, the result holds in every dimension $d\geq3$.  It hence recovers the
three-dimensional quintic and four-dimensional cubic models and also covers
all higher-dimensional energy-critical powers, which are non-algebraic.

\subsection{Some ideas behind the proof}
At the first glance, one may try to combine the deterministic bilinear estimate of
Kwak and Kwon established in \cite{KwakKwon} with a standard large deviation bound for the random linear
solution.  This is, however, not the case, since Kwak and Kwon's deterministic estimate does not retain
the positive power of the high frequency $N$ which is needed to sum the
random frequency pieces.  Recovering this gain is the main analytic issue of
the paper.  We briefly explain below how the new counting arguments enter and
why they lead to the gauge in \eqref{eq:gauged}.

\medskip
\noindent\emph{(i) The first counting and the mean-free coefficient.}
After a random block of frequency $N$ is decomposed into Galilean blocks of
width $R$, the relevant modulation contains the quantity $4Rk\cdot m$.
Here $k$ labels the Galilean block and $m$ is the Fourier frequency of the
low-frequency coefficient.  In the argument of Kwak and Kwon, one fixes $k$ and
counts the possible $m$.  For the random problem, however, the natural square
summation is over $k$, and this counting is too rough.  We instead fix
$m\neq0$ and count $k$.  This gives the slab estimate
\[
 \#\left\{k\in\mathbb Z^d:
 |k|_\infty\sim K,\ |4Rk\cdot m-\tau|\leq L\right\}
 \lesssim K^{d-1}\left(1+\frac{L}{R|m|}\right),
\]
where $K\sim N/R$.  The loss of one lattice dimension is precisely the extra
frequency gain needed in the random estimate.

There is one obvious exception.  When $m=0$, the modulation does not see
$k$, and the gain disappears.  Thus the zero Fourier mode of the coefficient
has to be removed.  This is the reason why the gauged nonlinearity $G$ appears
in our analysis.  In this sense, the gauge transform is forced by the
counting argument itself, in agreement with the mechanism found earlier by
Nahmod and Staffilani \cite{NahmodStaffilani15} in the algebraic energy-critical problem.

More precisely, the Bony increment of $G$ is written in the form
\[
 G(U+h)-G(U)
 =hA_0(U,h)+\overline h\,B(U,h)+R(U,h),
\]
where $A_0$ has spatial mean zero.  The term $hA_0$ is controlled by
combining the preceding slab count, the space-time Besov regularity of the
coefficient, and large deviation estimates for the random blocks.

\medskip
\noindent\emph{(ii) The opposite-phase term.}
The term containing $\overline h$ is different.  It contains two
Schr\"odinger factors with opposite phases, and the first Galilean shear no
longer applies.  We introduce another transform adapted to this interaction.
The corresponding modulation is
\[
 \eta+4Rk\cdot m-8R^2|k|^2.
\]
The quadratic term has size $N^2$ when $|k|\sim N/R$.  We separate the low-
and high-modulation regions relative to this scale.  In the low-modulation
region, the quadratic resonance gives a different lattice counting; in the
high-modulation region, temporal Besov regularity supplies the required
decay.  Together with large deviations, this yields the second randomized
bilinear estimate.  Thus the opposite-phase term is not a minor variation of
the first one: both the shear and the counting have to be changed.

\medskip
\noindent\emph{(iii) The scalar remainder.}
The final term $R(U,h)$ is scalar in space.  Spatial averaging separates
random variables which, in an algebraic expansion, would normally remain in
the same higher-order Wiener chaos.  Consequently, part of the usual chaos
cancellation is lost.  We deal with this term in a different way.  The
triangular frequency sum is reorganized as
\[
 \sum_NG_2^N=\sum_N\ell_N(t)y_N,
\]
and the coefficients $\ell_N$ are estimated by first-chaos large deviations,
space-time integrability, and a dyadic tail bound.  The small parameters
$\delta_0$ and $\vare_0$ are chosen so that the remaining frequency powers
are strictly summable.  This gives the positive smoothing in Theorem
\ref{main_thm}.  We point out that this scalar argument is independent of the
two Galilean countings above.

\medskip
\noindent\emph{(iv) The fixed point argument.}
Kwak and Kwon use an approximation and compactness argument in the regime
where the nonlinear coefficient is only H\"older continuous.  We shall
instead prove the theorem by a fixed point argument, following a strategy
similar to the one used in author's recent work \cite{LuoCriticalScattering}.  A direct fixed
point for $w=y-e^{it\Delta}\phi^\omega$, however, does not work.  Indeed, the
scalar part produces in the difference equation a term of the form
\[
 \bigl(\mu(y_1)-\mu(y_2)\bigr)e^{it\Delta}\phi^\omega.
\]
This term has the same rough regularity as the random linear solution and
therefore cannot be placed in the smoother space
$Y^{s_{\mathrm c}+\vare_0}$.

To isolate this rough one-dimensional direction, we introduce two unknowns:
a scalar function $\beta$ and a smoother remainder $v$.  We write
\[
 u=\beta z+v,
 \qquad
 y=z+\beta^{-1}v,
 \qquad z=e^{it\Delta}\phi^\omega
\]
and let $\beta$ absorb the scalar ODE contribution.  The contraction is then
measured through the reconstructed function
\[
 \mathcal J(\beta,v)=\beta z+v,
 \qquad
 d\big((\beta_1,v_1),(\beta_2,v_2)\big)
 =\|\mathcal J(\beta_1,v_1)-\mathcal J(\beta_2,v_2)\|_{Y^0}.
\]
The two components are not estimated separately.  This is important, since
the rough phase contribution and the smooth remainder may cancel after
reconstruction.  The resulting map intertwines exactly with the usual
Duhamel map for the original NLS.  At the fixed point, the scalar equation
gives $|\beta|=1$, and the gauged solution is recovered.  A localized version
of the same zero-order estimate then gives uniqueness in the class stated in
Theorem \ref{main_thm}.

\medskip
The rest of the paper is organized as follows.  In Section 2 we collect the
function spaces, the vector-valued Besov estimates, the Galilean transforms,
and the probabilistic tools used later.  In Section 3 we prove the
probabilistic space-time estimates and the two randomized bilinear estimates.
Section 4 is devoted to the Bony decomposition of the gauged nonlinearity and
the estimates for the mean-free, opposite-phase, and scalar terms.  Finally,
in Section 5 we construct the phase-adapted metric space and prove Theorem
\ref{main_thm} by the two-component contraction argument.

\section{Notation, definitions and auxiliary tools}
\subsection{Basic notation}
For $f:\T^d\to\C$ and $g:\R\times\T^d\to \C$, the spatial and global Fourier transforms are denoted by $\widehat f$ and $\widetilde g$ respectively. 

For a dyadic number $N\in 2^{\N_0}$, we use the sharp spatial Littlewood-Paley
projections
\[
 P_{\leq N}=P_{[-N,N]^d},
 \qquad
 P_N=P_{\leq N}-P_{\leq N/2},
\]
with $P_1=P_{\leq1}$. We denote by $P_C$ the spatial frequency cutoff projection for a given set $C\subset\Z^d$. Moreover, smooth
time-frequency projections are denoted by $P_L^t$. The operator $P_0$ defined by
\begin{align*}
P_0 f:=\fint_{\T^d}f(x)\,dx
\end{align*} 
is defined as the operator mapping a function to its spatial zero Fourier coefficient. 

For functions $u,v$ posed on $\R\times\T^d$, we define their paraproducts by
\begin{equation}\label{paraproducts}
\begin{aligned}
\pi_{\leq }(u,v)
:=
\sum_{N\ge M/16}
u_Mv_N,
\qquad
\pi_{>}(u,v)
:=
\sum_{M\ge 32N}
u_Mv_N.
\end{aligned}
\end{equation}

$K^+$ denotes the Schr\"odinger operator
$$ K^+ f(t,x):=\int_0^t e^{i(t-s)\Delta}f(s,x)\,ds.$$
Moreover, a pair $q,r\in[2,\infty]$ is said to be {$\Delta$-admissible} if $\frac{2}{q}+\frac{d}{r}=\frac{d}{2}-\Delta$.

For an interval $I\subset\R$, the function $\chi_I$ denotes the corresponding sharp characteristic function. The function $\psi$ denotes a smooth bump function supported on $[-1,1]$.

\subsection{Dictionary of constants and dependency order}\label{dictionary}
For the reader's convenience, we give here a collection of the parameters which will be used repeatedly in the paper. 

\begin{itemize}
\item Choose
\begin{equation*}
 0<\sigma\ll\sigma_1\ll\sigma_2\ll\sigma_3\ll\sigma_4
 \ll 1
\end{equation*}
and define the numbers $p,q_0,r_0,\theta$ by
\begin{equation}\label{eq:constant-dictionary-KK}
\begin{aligned}
 \frac{1}{p}=\frac{d/2-\sigma}{d+2},\quad
 \frac1{q_0}=\frac{2+\sigma_3}{d+2},\quad
 \frac1{r_0}=\frac{2+\sigma_3+\sigma_2}{d+2},
\quad
 \theta=\frac2{q_0}+\frac d{r_0}-2.
\end{aligned}
\end{equation}
\item
The numbers $\alpha,\beta,\tilde\alpha,\tilde\beta$ are defined by
\begin{equation}\label{eq:constant-dictionary-alpha}
\begin{aligned}
 \frac1{q_0}+\sigma_1+2\left(\frac1p-\alpha\right)&=1,\qquad
 \beta=\sigma+2\alpha,\\
 \frac1{q_0}+2\left(\frac1p-\tilde\alpha\right)&=1,\qquad
 \tilde\beta=\sigma+2\tilde\alpha.
\end{aligned}
\end{equation}
\item 
The numbers $\widehat{r},\zeta,\eta$ are defined by
\begin{equation*} 
\begin{aligned}
\frac1{\widehat r}=\frac{1+\sigma_4}{d+2},\quad
 \zeta=\frac1{\widehat r}-\frac1{2q_0},
 \quad \eta=\frac d{\widehat r}-\frac d{2r_0}+\frac\theta2.
\end{aligned}
\end{equation*} 
Define also the coefficient factorization
\begin{align*}
J=\min\{j\in\mathbb N:j>a/2\},\quad
 a_*=\frac{a}{2J}.
\end{align*}
The number $s_0,s_1,r$ are chosen such that $\sigma_4\ll s_0\ll 1$, $\min\{s_0 a_*-\zeta,s_1a_*-\eta\}>0$, 
\begin{align*}
\frac{a}{2}\left(\frac{1}{r}-s_0\right)
=
\frac{1}{\hat r}
-\zeta
\qquad\text{and}\qquad
\frac{a}{2}\left(\frac{1}{r}-\frac{s_1}{d}\right)
=
\frac{1}{\hat r}
-\frac{\eta}{d}.
\end{align*}

\item The numbers $0<\vartheta\ll 1$ and $\tilde\vartheta\in(0,1)$ are constants determined by Remark \ref{vartheta explain}.

\item The number $\vare>0$ denotes the order of the derivative loss $N^\vare$ coming from the random estimates.

\item The number $\alpha_{\mathrm op}$ is defined via
\begin{equation}\label{2.29}
\left(\frac{1}{p}-\frac{3\sigma_1}{2}\right)+\left(\frac{1}{p}-\alpha_{\mathrm op}\right)
+\left(\frac{1}{q_0}+\sigma_1\right)=1.
\end{equation}
 
\item The numbers $\nu$ and $\mu$ are defined by
\begin{equation*}
 \nu=1+a-s_c-\sigma_1>0,
 \qquad
 \mu=\left(1-\frac{2\sigma}{\sigma_1}\right)(\nu+\sigma)
 +\frac{2\sigma}{\sigma_1}(-s_c+\sigma)>0.
\end{equation*}
The numbers $q_h,r_h$ are defined by
\begin{equation}\label{eq:qh-rh}
 \frac{1+a}{q_h}=\frac1{p'},
 \qquad
 \frac2{q_h}+\frac d{r_h}=\frac d2-\sigma.
\end{equation}

\item The number $\alpha_a$ is defined by $\alpha_a=\min\{1,a\}$.

\item Choose $\delta_0,\vare_0>0$ sufficiently small so that
\begin{equation*}
\begin{gathered}
 2s_0+s_1<s_c-\delta_0,\qquad
 0<\vare_0<\min\{\mu,\nu,1+a-s_c\},\\
 \delta_0+\vare_0<\sigma,\qquad
 (1+\alpha_a)\delta_0+\vare_0<\alpha_a s_c.
\end{gathered}
\end{equation*}
After fixing $\delta_0$ and $\vare_0$, choose the derivative loss $\vare>0$ so small that
\begin{equation}\label{eq:delta-epsilon-varepsilon-margins}
 2\vare<\sigma-\delta_0-\vare_0,\qquad
 (1+\alpha_a)(\delta_0+\vare)+\vare_0<\alpha_a s_c.
\end{equation}
For $0<\delta<\delta_0$, set
\begin{equation}\label{eq:star-parameters}
 s^*=s_c+\vare_0,\qquad
 \nu^*=\nu-\vare_0,\qquad
 \mu^*=\mu-\vare_0,\qquad
 \sigma_\delta=\sigma-\delta-\vare_0-2\vare.
\end{equation}

\item For $a>0$, we define the number $r_{\text{rem}}$ by 
\begin{align}\label{p>2}
r_{\text{rem}}=2\max\{1,a-1\}.
\end{align}
After decreasing $\sigma$ if necessary, a number $q_{\text{rem}}$ can be chosen to satisfy
\begin{equation}\label{eq:qg-choice}
\begin{aligned}
 &\max\{p,a+1,2\}<q_{\text{rem}}<\frac1\sigma,
 \qquad
 \frac d{r_{\text{rem}}}+\frac2{q_{\text{rem}}}>\frac2a\\
&1-\frac{a}{q_{\mathrm{rem}}}
 >\max\left\{s_0+\frac{s_1}{2},\frac1p-\sigma\right\}. 
 \end{aligned}
\end{equation}
For later use, we also define the numbers
\begin{equation*}
 p_{\rm rem}:=\frac{q_{\mathrm{rem}}}{a}>1,
 \qquad
 \gamma_{\mathrm{rem}}
 :=1-\frac1{p_{\rm rem}}
 =1-\frac{a}{q_{\mathrm{rem}}}.
\end{equation*}
\end{itemize}

\subsection{Vector-valued Besov and Sobolev spaces}
Given a Banach space $E$, the numbers $k\in\N$, $s\in\R$, $p,q\in[1,\infty]$ and the argument $z\in\{\R^k,\T^k\}$, the vector-valued Banach space $B_{z,p,q}^s$ is defined via the norm
$$ \|u\|_{B_{z,p,q}^s E}:=\bg(\sum_{N\in 2^\N}N^{qs}\|P_N^z(u)\|_{L_z^p E}^q\bg)^{\frac{1}{q}}+\|P_1^z u\|_{L_z^p E}.$$
For our purpose, we will alternatively make use of the following well-known characterization of the vector-valued Besov spaces.

\begin{lemma}[Difference characterization of Besov spaces, \cite{chara_besov}]\label{lem_triebel}
Define
$$ \Delta_h f(x):=f(x+h)-f(x).$$
Let also $E$ be a Banach space. Then for $s\in(0,1)$, $k\in\N$ and $1\leq p,q\leq \infty$ we have
\[\|u\|_{B_{p,q}^s(\R^k;E)}\sim \|u\|_{L^p(\R^k;E)}
+\left(\int_0^1 \left(h^{-s}\|\Delta_h f\|_{L^p(\R;E)}\right)^q\frac{dh}{h^{k}}\right)^{\frac1q}.\]
\end{lemma}

The following lemma gives the fundamental properties of the vector-valued Besov spaces.
\begin{lemma}[Embedding, duality and interpolation for Besov spaces, \cite{amannbesov,AmannEmbedding,NakamuraWada}]\label{lemallprop}
Let $E$, $E_1$, $E_2$ be Banach spaces and $B_{p,q}^s$ denote Besov-spaces defined on $\R^k$ with values in $E$. For given numbers $p_i,q_i,s_i$, $i=0,1$, define
$p_\theta^{-1}=(1-\theta)p_0^{-1}+\theta p_1^{-1}$, $q_\theta^{-1}=(1-\theta)q_0^{-1}+\theta q_1^{-1}$ and $s_\theta=(1-\theta)s_0+\theta s_1$. Then the following statements hold true:
\begin{itemize}
\item[(i)] For $m\in\Z$ and $p\in[1,\infty)$ we have $B_{p,1}^m E\hookrightarrow W^{m,p} E\hookrightarrow B_{p,\infty}^m E$.

\item[(ii)] For $1\leq p_1<p_2<\infty$ and $q\in[1,\infty]$ we have $B_{p_1,q}^{k(\frac{1}{p_1}-\frac{1}{p_2})} E\hookrightarrow L^{p_2,q} E$.

\item[(iii)] For $1\leq p_1<p_2<\infty$ and $M\in 2^{\N_0}$ we have $\|P_M f\|_{L^{p_2} E}\lesssim M^{k(\frac{1}{p_1}-\frac{1}{p_2})}\|P_M f\|_{L^{p_1} E}$.
\item[(iv)] For $1\leq p_1<p_2<\infty$, $q\in[1,\infty]$ and $s\in\R$ we have $B_{p_1,q}^{s+k(\frac{1}{p_1}-\frac{1}{p_2})} E\hookrightarrow B_{p_2,q}^{s} E$.

\item[(v)] For $p,q\in[1,\infty)$, $s\in\R$, and either $E$ being reflexive or $E'$ being separable\footnote{A Banach space satisfying such property is referred to as a Banach space satisfying the \textit{Radon-Nikodym property} in literature.}, we have $(B^s_{p,q}E)'=B^{-s}_{p',q'}E'$.

\item[(vi)] For $p\in[1,\infty)$, $q_0,q_1,\eta\in[1,\infty]$, $\theta\in(0,1)$ and $s_0,s_1\in\R$ with $s_0\neq s_1$ we have the real interpolation $(B_{p,q_0}^{s_0} E, B_{p,q_1}^{s_1} E)_{\theta,\eta}=B_{p,\eta}^{s_\theta}E$.
\item[(vii)] For $p_0,p_1\in(1,\infty)$, $q_0,q_1\in[1,\infty]$, $s_0,s_1\in\R$ and an interpolation couple $(E_0,E_1)$ we have the complex interpolation
$[B_{p,q_0}^{s_0} E_0, B_{p,q_1}^{s_1} E_1]_{\theta}=B_{p_\theta,q_\theta}^{s_\theta}[E_0,E_1]_\theta$.
\end{itemize}
\end{lemma}

We will also make of Sobolev spaces defined via the Bessel potential. Given $p\in(1,\infty)$ and $s\in\mathbb{R}$, we denote by
$H^{s,p}(\mathbb{T}^d)$ the (fractional regularity) Sobolev space
given by the norm
$\|f\|_{H^{s,p}}
=
\|\mathcal{F}_x^{-1}(\widehat{f}(\xi)\cdot\langle\xi\rangle^s)\|_{L^p(\mathbb{T}^d)}$,
where $\langle\xi\rangle=\sqrt{1+|\xi|^2}$.

\subsection{Fractional calculus in Besov spaces}\label{sec 3.2}
We record some useful auxiliary tools from \cite{KwakKwon} involving fractional calculus for periodic functions.

\begin{lemma}[Banach space-valued Schur-Besov inequality]\label{kwak_lem_lhh}
Let $E_i$, $i\in\{1,2,3\}$, be Banach spaces defined on a measure space $\mathcal{M}$ and assume the inequality
\begin{align*}
\bg|\int_{\mathcal{M}} f_1f_2f_3\,dm\bg|\lesssim \|f_1\|_{E_1}\|f_2\|_{E_2} \|f_3\|_{E_3}
\end{align*}
holds. Let $s_j\in\mathbb{R}$,
$p_j\in(1,\infty)$,
$q_j\in[1,\infty]$,
$j=1,2,3$
be parameters such that
\[
s_1+s_2+s_3>0,\qquad
\frac1{p_j}>s_j,\qquad
\frac1{p_1}+\frac1{p_2}+\frac1{p_3}
=
s_1+s_2+s_3+1,
\]
and
\[
\frac1{q_1}+\frac1{q_2}+\frac1{q_3}=1.
\]

\begin{enumerate}
\item
Assume
$s_1+s_2>0$,
$s_1+s_3>0$,
and
$s_2+s_3>0$.
Then, we have
\begin{equation*}
\|uv\|_{B^{-s_3}_{p_3',q_3'}E_3'}
\lesssim
\|u\|_{B^{s_1}_{p_1,q_1}E_1}
\|v\|_{B^{s_2}_{p_2,q_2}E_2}.
\end{equation*}

\item
Assume
$s_1+s_2>0$
and
$s_2+s_3>0$.
Then, we have
\begin{equation}\label{2.14}
\|\pi_{\le}(u,v)\|_{B^{-s_3}_{p_3',q_3'}E_3'}
\lesssim
\|u\|_{B^{s_1}_{p_1,q_1}E_1}
\|v\|_{B^{s_2}_{p_2,q_2}E_2}.
\end{equation}
\end{enumerate}
\end{lemma}

\begin{lemma}[First fractional chain rule periodic functions]\label{lem3.9-}
Let $\alpha \in (0,1)$ and $F \in C^{0,\alpha}(\mathbb{C})$. Let also
$s \in (0,\alpha)$, $\sigma\in(s/\alpha,\infty)$, and
$p,p_1,p_2 \in (1,\infty)$ be exponents satisfying
\(
\frac{1}{p}
=
\frac{1}{p_1}
+
\frac{1}{p_2}
\)
and
\(
\left(1-\frac{s}{\alpha\sigma}\right)p_1>1.
\) Then we have
\begin{equation*}
\|F(u)\|_{H^{s,p}}
\lesssim
\|u\|_{L^{(\alpha-\frac{s}{\sigma})p_1}}^{\,\alpha-\frac{s}{\sigma}}
\cdot
\|u\|_{H^{\sigma,\frac{s}{\sigma}p_2}}^{\,\frac{s}{\sigma}}.
\end{equation*}
\end{lemma}

\begin{lemma}[Second fractional chain rule for periodic functions]\label{lem3.9}
Let $\alpha\in(1,\infty)$, $s\in[0,\alpha)$ and $k\in\Z$. Let also $p,p_1,p_2\in(1,\infty)$ satisfy $p^{-1}=(\alpha-1)p_1^{-1}+p_2^{-1}$. Then for $u:\T^n\to \C$ we have
\begin{align*}
\||u|^{\alpha-k} u^k\|_{H^{s,p}_y}\lesssim \|u\|_{L_y^{p_1}}^{\alpha-1}\|u\|_{H_y^{s,p_2}}.
\end{align*}
\end{lemma}

\begin{lemma}[First H\"older-Besov inequality]
Fix $s\in(0,1)$, $p\in(1,\infty)$, $\alpha\in(0,1)$, and
$F\in C^{0,\alpha}(\mathbb{C})$. For
$u\in B^{s}_{p,p}(\mathbb{T}^d)$, we have
\begin{equation*}
\|F(u)\|_{B^{s\alpha}_{p/\alpha,p/\alpha}(\mathbb{T}^d)}
\lesssim
\|u\|_{B^{s}_{p,p}(\mathbb{T}^d)}^{\alpha}.
\end{equation*}
\end{lemma}

\begin{lemma}[Second H\"older-Besov inequality]
Let $s_0,s_1>0$ be exponents satisfying
$2s_0+s_1<1$.
Fix $p\in(1,\infty)$, $\alpha\in(0,1)$, and a function
$F\in C^{0,\alpha}(\mathbb{C})$.
For $u:\mathbb{R}\times\mathbb{T}^d\to\mathbb{C}$, we have
\begin{equation*}
\|F(u)\|_{B^{s_0\alpha}_{p/\alpha,p/\alpha}
B^{s_1\alpha}_{p/\alpha,p/\alpha}}
\lesssim
\|u\|_{L^p
B^{2s_0+s_1}_{p,p}
\cap
B^{s_0+s_1/2}_{p,p}
L^p}^{\alpha}.
\end{equation*}
\end{lemma}

\subsection{Atomic and $Z^s$-spaces}\label{sec_zs_space}
We introduce in this subsection the atomic spaces introduced in \cite{HerrTataruTz1,HerrTataruTz2,HaniPausader,RmT1} as well as the $Z^s$ space in \cite{KwakKwon} which will be used to study the Cauchy problem \eqref{eq:nls}.

\subsubsection{The atomic spaces}
Let $H$ be a separable Hilbert space. Let $\mathcal Z$ be the collection
of finite non-decreasing sequences $\{t_k\}_{k=0}^K$ in
$(-\infty,\infty]$. For $1\le p<\infty$, we call
$a:\mathbb R\to H$ a $U^p$-atom if $a$ can be expressed as
\[
a=\sum_{k=1}^K\chi_{[t_{k-1},t_k)}\phi_k,
\qquad
\sum_{k=1}^K\|\phi_k\|_H^p=1.
\]
We define $U^pH$ as the space of all functions
$u:\mathbb R\to H$ that can be represented as
\[
u=\sum_{j=1}^\infty\lambda_ja_j,
\]
where $a_j$ is a $U^p$-atom for each $j\in\mathbb N$ and
$\{\lambda_j\}\in\ell^1$ is a complex-valued sequence, equipped with the
norm
\[
\|u\|_{U^pH}
:=
\inf
\left\{
\sum_{j=1}^\infty|\lambda_j|
:
u=\sum_{j=1}^\infty\lambda_ja_j,\,
\lambda_j\in\mathbb C,\,
a_j:\text{$U^p$-atom}
\right\}.
\]

We define $V^pH$ as the space of all functions
$u:\mathbb R\to H$ with $\|u\|_{V^pH}<\infty$, where the norm is defined
as
\[
\|u\|_{V^pH}^p
:=
\sup_{\{t_k\}_{k=0}^K\in\mathcal Z}
\sum_{k=1}^K
\|u(t_k)-u(t_{k-1})\|_H^p,
\]
where the convention $u(\infty)=0$ is used. Then, we define
$V_{rc}^pH$ as the subspace of $V^pH$ of right-continuous function
$u:\mathbb{R}\to H$ satisfying
\(
\lim_{t\to-\infty}u(t)=0.
\)
For simplicity of notation, we omit $H$ in
$U^pH$, $V^pH$, $V_{rc}^pH$ when $H\simeq\mathbb{C}$.
Based on this, we define the spaces
\(
U_\Delta^pH,
V_\Delta^pH,
V_{\Delta,rc}^pH
\)
as the images by the map
\(
u\mapsto e^{it\Delta}u
\)
of
\(
U^pH,
V^pH,
V_{rc}^pH,
\)
respectively.

Denote by $C=(-1/2,1/2]^d\in \R^d$ the unit cube in $\R^d$. For $z\in\R^d$ the translated cube $C_z$ is defined by $C_z:=C+z$. Moreover, we define the projector $P_{C_z}$ by
$$\mathcal{F}(P_{C_z}u):=\chi_{C_z}\mathcal{F}(u),$$
where $\chi_{C_z}$ is the characteristic function of $C_z$. For $s\in\R$ we then define the spaces $X_0^s(\R)$ and $Y^s(\R)$ through the norms
\begin{align}\label{def XsYs}
\begin{aligned}
\|u\|_{X^s_0(\R)}^2:=\sum_{z\in\Z^d}\la z\ra^{2s}\|P_{C_z} u\|_{U_{\Delta}^2(\R;L_x^2)},\\
\|u\|_{Y^s(\R)}^2:=\sum_{z\in\Z^d}\la z\ra^{2s}\|P_{C_z} u\|_{V_{\Delta}^2(\R;L_x^2)}.
\end{aligned}
\end{align}
For any subinterval $I\subset\R$, the space $X^s(I)$ is defined through the norm
\begin{align*}
\|u\|_{X^s(I)}:=\inf\{\|v\|_{X^s_0(\R)}:v\in X_0^s(\R),\,v|_I=u|_I\}.
\end{align*}
The space $Y^s(I)$ is similarly defined. 

We record the following useful properties of the previously defined function spaces.
\begin{lemma}[Embeddings between function spaces, \cite{HadacHerrKoch2009,HerrTataruTz1}]
For any $s\in\R$ and $p\in(2,\infty)$ we have
\begin{align*}
U^2_{\Delta}(I;H^s_{x})\hookrightarrow X^s(I)\hookrightarrow Y^s(I)\hookrightarrow V_{\Delta}^2(I;H_{x}^s)
\hookrightarrow U^p_{\Delta}(I;H_{x}^s)\hookrightarrow L_t^\infty(I;H_{x}^s).
\end{align*}
\end{lemma}

\begin{lemma}[Duality, \cite{HerrTataruTz1}]\label{basis duhamel}
For $u\in L_t^1H_{x}^{s}(I)$ we have
\begin{align*}
\|K^+u\|_{X^{s}(I)}\lesssim \sup_{\|v\|_{Y^{-s}(I)}\leq 1}\int_{I\times(\T^d)}u(t,x)\bar{v}(t,x)\,dxdt.\lesssim \|u\|_{L_t^1 H_x^{s}}.
\end{align*}
\end{lemma}

\subsubsection{The $Z^s$-space}
We next give the definition of the $Z^s$-space and record its useful properties established in \cite{KwakKwon}. Fix the regularity number $s\in\R$. For
\begin{equation*}
 0<\sigma\ll1,
 \qquad
 p=\frac{d+2}{d/2-\sigma},
\end{equation*}
the space $Z^s$ is defined via the following norm:
\begin{equation}\label{def Zs}
\begin{aligned}
\|u\|_{Z^s}
={}&
\max_{q\in\left[p,\frac{1}{\sigma}\right]}
\left\|
\left\|
\psi' u_N
\right\|_{L^qL^r}
\right\|_{\ell^{2,s-\sigma}_N}
\\
&\quad+
\max_{\alpha\in\left[\sigma,\frac1p-\sigma\right]}
\left\|
\max_{\substack{R\in2^{\mathbb N}\\ R\le 8N}}
R^{-(2\alpha+\sigma)}
\left\|
\left\|
\psi P_{\le 8R}I_{Rk}u_N
\right\|_{B^{\alpha}_{p,1}L^p}
\right\|_{\ell^2(k\in\mathbb Z^d)}
\right\|_{\ell^{2,s}_N}.
\end{aligned}
\end{equation}

\begin{lemma}\label{lem:3.6}
We have the following properties:
\begin{itemize}
\item
For a finite interval $I\subset\mathbb{R}$, we have the embedding
\begin{equation}\label{3.11}
\ell_s^2(Z^0)'
=
(Z^{-s})'
\hookrightarrow
(Y^{-s})'
\xrightarrow{K^+}
Y^s
\hookrightarrow
Z^s
=
\ell_s^2 Z^0.
\end{equation}

\item
For a finite interval $I\subset\mathbb{R}$, we have
\begin{equation*}
\|u\cdot\chi_I\|_{Z^s}
\lesssim
\|u\|_{Z^s}.
\end{equation*}
This estimate is uniform in the choice of $I$.

\item
For $u\in Z^s$, we have
\begin{equation}\label{3.13}
\lim_{T\to0^+}
\|u\cdot\chi_{[0,T]}\|_{Z^s}
=
0.
\end{equation}

\item
Let
$q\in\left[p,\frac1\sigma\right]$
and $r$ be parameters such that
\[
\frac2q+\frac dr=\frac d2-\sigma.
\]
We have
\begin{equation}\label{3.14}
\|\psi u\|_{L^qH^{s-\sigma,r}}
\lesssim
\|\psi u\|_{L^qB^{\,s-\sigma}_{r,2}}
\lesssim
\|\psi u\|_{\ell^2_{\,s-\sigma}L^qL^r}
\lesssim
\|u\|_{Z^s}.
\end{equation}

\item
Let
$\alpha\in\left[\sigma,\frac1p-\sigma\right]$
and
$\beta=\sigma+2\alpha$.
We have
\begin{equation}\label{3.15}
\|\psi u\|_{B^{\alpha}_{p,2}B^{\,s-\beta}_{p,2}}
\lesssim
\|\psi u\|_{\ell^2_{\,s-\beta}B^{\alpha}_{p,1}L^p}
\lesssim
\|u\|_{Z^s}.
\end{equation}
\end{itemize}
Similar properties hold with $s$ replaced by $0$.
\end{lemma}

We also record a useful multi-product variant of Lemma \ref{kwak_lem_lhh} which was originally shown and applied in the proof of \cite[Prop. 3.11.]{KwakKwon}. 

\begin{lemma}\label{lem:mixed-coeff}
For $m\in \Z$ we have
\begin{equation*}
 \|\psi|u|^{\frac{a}{2}-m}u^m\|_{
 B^{\zeta}_{t,\widehat r,\widehat r}
 B^{\eta}_{x,\widehat r,\widehat r}}
 \lesssim\|\psi u\|_{L_t^rB^{2s_0+s_1}_{x,r,2}
\cap
B^{s_0+s_1/2}_{t,r,2}L_x^r}^{a/2}
 \lesssim \|u\|_{Z^{s_c}}^{a/2}
\end{equation*}
and
\begin{equation*}
 \|\psi|u|^{a-m}u^m\|_{
 B^{\frac{1}{r_0}-\frac{1}{q_0}}_{t,r_0,r_0}
 B^{\theta}_{x,r_0,r_0}}
  \lesssim\|\psi u\|_{L_t^rB^{2s_0+s_1}_{x,r,2}
\cap
B^{s_0+s_1/2}_{t,r,2}L_x^r}^{a}
 \lesssim \|u\|_{Z^{s_c}}^{a},
\end{equation*}
where the appearing parameters are defined in Section \ref{dictionary}.
\end{lemma}

\subsection{The Galilean transform}
We shall establish suitable bilinear estimates within the space $Z^s$, along with the Galilean transform
\begin{align}\label{Galilean trans}
I_{\xi}u(t,x)=e^{ix\cdot\xi-it|\xi|^2}u(t,x-2t\xi).
\end{align}
Here we also record some properties of the Galilean transform which will be repeatedly used throughout the paper. 
\begin{lemma}\label{lem 3.8}
For any $\xi\in \R^d$ the following properties of the Galilean transformation hold.
\begin{itemize}
\item $(i\pt_t +\Delta)I_\xi u=I_\xi (i\pt_t +\Delta) u$.
\item For any set $C\subset \Z^d$ it holds $P_{C+\xi}I_\xi u= I_\xi P_C u$.
\item $\|I_\xi u\|_{Y^0}=\|u\|_{Y^0}$.
\end{itemize}
\end{lemma} 

\subsection{Probabilistic tools}
In this final subsection we record some useful auxiliary large deviation estimates.
\begin{lemma}[Almost sure endpoint roughness, \cite{BurqTzvetkov1}]
\label{lem:endpoint-roughness}
For the random Fourier series \eqref{eq:random-data}, it holds $\phi^\omega\notin H^{s_c}$ almost surely.
\end{lemma}

\begin{lemma}[Large deviation, \cite{BurqTzvetkov1}]\label{lem_large_dev}
Let $(l_n(\omega))_{n=1}^\infty$ be a sequence of real, independent random
variables with associated sequence of distributions
$(\mu_n)_{n=1}^\infty$. Assume that $\mu_n$ satisfy the property
\begin{equation*}
\exists\, c>0:\quad
\forall \gamma\in\mathbb{R},\ \forall n\ge1,\quad
\left|
\int_{-\infty}^{\infty} e^{\gamma x}\,d\mu_n(x)
\right|
\le e^{c\gamma^2}.
\end{equation*}
Then there exists $\alpha>0$ such that for every $\lambda>0$, every sequence
$(c_n)_{n=1}^\infty\in\ell^2$ of real numbers,
\begin{equation*}
p\left(
\omega:
\left|
\sum_{n=1}^{\infty} c_nl_n(\omega)
\right|
>\lambda
\right)
\le
2e^{
-\frac{\alpha\lambda^2}
{\sum_n c_n^2}
}.
\end{equation*}
As a consequence there exists $C>0$ such that for every $p\ge2$, every
$(c_n)_{n=1}^\infty\in\ell^2$,
\begin{equation*}
\left\|
\sum_{n=1}^{\infty} c_nl_n(\omega)
\right\|_{L^p(\Omega)}
\le
C\sqrt{p}
\left(
\sum_{n=1}^{\infty}c_n^2
\right)^{1/2}.
\end{equation*}
\end{lemma}

\begin{lemma}[Almost sure finiteness, \cite{Tzvetkov10}]\label{2.4 lem}
Let $F$ be a measurable function and suppose that there exist $C_0,K > 0$ and $p_0 \geq 1$ such that for any $p \geq p_0$ we have
\[\|F\|_{L_\omega^p}\leq C_0\sqrt{p}K.\]
Then there exist $c,C_1>0$, depending on $C_0$ and $p_0$ but not on $K$, such that for any $\ld>0$ we have
\[
\mathbb{P}\bg(\{\omega\in\Omega:|F(\omega)|>\ld\}\bg)\leq C_1\exp\{-c\ld^2 K^{-2}\}.
\]
\end{lemma}
\section{Randomized bilinear estimates}\label{sec 333}
The purpose of this section is to convert the probabilistic gain of the
random linear solution into a frequency gain compatible with the
Kwak and Kwon's Galilean framework.  Two different interaction geometries arise.
The first estimate uses the zero-mean property of the coefficient and a
codimension-one lattice count.  The second treats the conjugated, or
opposite-phase, interaction by means of a renormalized shear.  Both estimates
combine pathwise Gaussian bounds with space-time Besov regularity.

\subsection{Probabilistic space-time Besov estimates}
\begin{lemma}\label{lem:time-cutoff}
Let $2\leq q<\infty$, $0<\alpha<1/q$, $0<T<\frac12$, and
$I_T=[-T,T]$.  Let $\psi$ be smooth function satisfying $\|\psi\|_{C^1(\R)}\lesssim 1$. Then for any $\kappa\in\mathbb R$ we have 
\begin{equation}\label{eq:time-cutoff-Besov}
 \|\psi\chi_{I_T}e^{-it\kappa}\|_{B^\alpha_{q,2}(\mathbb R)}
 \lesssim_{\alpha,\psi} T^{1/q-\alpha}\langle\kappa\rangle^\alpha.
\end{equation}
\end{lemma}

\begin{proof}
By Lemma \ref{lem_triebel} we have the equivalent Besov norm via the difference characterization:
\begin{equation*}
 \|f\|_{B_{q,2}^\alpha}\sim\|f\|_{L^q}
 +\left(\int_0^1h^{-2\alpha}
 \|f(\cdot+h)-f\|_{L^q}^2\frac{dh}{h}\right)^{1/2}.
\end{equation*}
It is straightforward to verify that $\|\psi\chi_{I_T}e^{-it\kappa}\|_{L^q}\lesssim T^{\frac1q}$. For the second part, we separate the sum to
\[\int_0^1=\int_0^{2T}+\int_{2T}^1.\]
Setting $f=\psi\chi_{I_T}e^{-it\kappa}$, we see that the supports of $f(\cdot+h)$ and $f$ are disjoint for $h\in(2T,1)$, hence
\[\left(\int_{2T}^1h^{-2\alpha}
\|f(\cdot+h)-f\|_{L^q}^2\frac{dh}{h}\right)^{\frac12}
 \lesssim T^{\frac1q}\cdot \left(h^{-2\alpha}|_{2T}^1\right)^{\frac12}\lesssim T^{\frac1q-\alpha}.\]
Next, set 
$$I_1=(-T-h,-T)\cup (T-h,T),\qquad I_2=(-T,T-h).$$
For $h\in(0,2T)$ and $t\in I_1$, only one of the summand in the difference $f(\cdot+h)-f$ is non-zero. Hence
\[\left(\int_0^{2T}h^{-2\alpha}
\|f(\cdot+h)-f\|_{L^q(I_1)}^2\frac{dh}{h}\right)^{\frac12}
 \lesssim 
\left(\int_0^{2T}h^{2(\frac1q-\alpha)-1}\,dh\right)^{\frac12}
\lesssim T^{\frac1q-\alpha}.\]
For $h\in(0,2T)$ and $t\in I_2$, telescoping yields
\begin{align*}
 f(t+h)-f(t)
 =e^{-i\kappa(t+h)}\big(\psi(t+h)-\psi(t)\big)
 +e^{-i\kappa t}\big(e^{-i\kappa h}-1\big)\psi(t).
\end{align*}
The first part contributes 
$$ \lesssim (2T-h)^{\frac1q}\left(\int_0^{2T}h^{-2\alpha}h^2\,\frac{dh}{h}\right)^{\frac12}\lesssim T^{1+\frac1q-\alpha}\lesssim T^{\frac1q-\alpha}.$$
The second part contributes
\[
(2T-h)^{\frac1q}\left(\int_0^{2T}h^{-2\alpha}
 \min\{1,|\kappa|h\}^2\frac{dh}{h}\right)^{\frac12}
 \lesssim T^{1/q}|\kappa|^{\alpha}.
\]
\eqref{eq:time-cutoff-Besov} follows from collecting the resulting estimates.
\end{proof} 

\begin{lemma}[Gaussian Besov--Strichartz estimate]\label{lem:prob-Besov}
Let $q,r\in[2,\infty)$, $\alpha_t\in(0,1/q)$,
$\alpha_x\in\R$, and $I_T=[-T,T]$, $0<T\leq\frac12$.  Then for every
$p_\omega\geq2$ it holds
\begin{equation}\label{eq:prob-Besov}
 \|\psi\chi_{I_T}e^{it\Delta}\phi^\omega\|_{L_\omega^{p_\omega}
 B^{\alpha_t}_{t,q,2}B^{\alpha_x}_{x,r,2}}
 \lesssim
 \sqrt{p_\omega}\,T^{1/q-\alpha_t}\|\phi\|_{H^{2\alpha_t+\alpha_x}}.
\end{equation}
Moreover, for $\vartheta\in(0,\frac{1}{q}-\alpha_t)$ there exist $C,c>0$ and $\Omega_T\subset \Omega$ satisfying $\mathbb P(\Omega_T^c)<\exp(-cT^{
2(\vartheta-(\frac{1}{q}-\alpha_t))})$ such that for any $\omega\in\Omega_T$ it holds
\begin{equation}\label{eq:prob-pathwise-time}
 \|\psi\chi_{I_T}e^{it\Delta}\phi^\omega\|_{
 B^{\alpha_{t}}_{q,2}B^{\alpha_{x}}_{r,2}}
 \leq C T^\vartheta \|\phi\|_{H_x^{2\alpha_t+\alpha_x}}.
\end{equation}
\end{lemma}

\begin{proof}
We only prove the claim in the case $B^{\alpha_x}_{x,r,2}$, the other cases can be similarly shown. For $p_\omega\geq \max\{q,r\}$, Lemma \ref{lem_large_dev}, Minkowski and the compactness of $\T^d$ yield
\begin{equation*}
\begin{aligned}
 &\|P_L^tP_N(\psi\chi_{I_T}e^{it\Delta}\phi^\omega)\|_{
 L_\omega^{p_\omega}L_t^qL_x^r}
 =\|P_L^tP_N(\psi\chi_{I_T}b_ng_n(\omega)e^{i(n\cdot x-|n|^2t)})\|_{
 L_\omega^{p_\omega}L_t^qL_x^r}
 \\
 &\quad
 \lesssim \left\|\left\|\sum_{|n|_\infty\sim N}g_n(\omega)b_n e^{in\cdot x}P_L^t (\psi\chi_{I_T}e^{-i|n|^2t})\right\|_{L_\omega^{p_\omega}}\right\|_{
L_t^qL_x^r }
 \\
 &\quad\lesssim\sqrt{p_\omega}
 \left\|
 \left(\sum_{|n|_\infty\sim N}|b_n|^2
 |P_L^t(\psi\chi_{I_T}e^{-it|n|^2})|^2\right)^{1/2}
 \right\|_{L_t^q}.
\end{aligned}
\end{equation*}
Taking the $\ell_L^2$ norm
with weight $L^{\alpha_t}$ and using Lemma \ref{lem:time-cutoff} and Minkowski give
\begin{align*}
 &\|P_N(\psi\chi_{I_T}e^{it\Delta}\phi^\omega)\|_{L_\omega^{p_\omega}
 B^{\alpha_t}_{t,q,2}L_x^r(I_T)}\lesssim
 \|L^{\alpha_t}P_L^tP_N(\psi\chi_{I_T}e^{it\Delta}\phi^\omega)\|_{
 \ell_L^2 L_\omega^{p_\omega}L_t^qL_x^r}\notag\\
 &\qquad\lesssim
 \sqrt{p_\omega}
 \left(\sum_{|n|_\infty\sim N}|b_n|^2
 \|\psi\chi_{I_T}e^{-it|n|^2}\|_{B_{t,q,2}^{\alpha_t}}^2\right)^{1/2}
\lesssim\sqrt{p_\omega}\,T^{1/q-\alpha_t}
 \left(\sum_{|n|_\infty\sim N}|n|^{4\alpha_t}|b_n|^2\right)^{1/2}.
\end{align*}
Multiplication by $N^{\alpha_x}$ and followed by an $\ell_N^2$ summation prove
\eqref{eq:prob-Besov} in the case $p_\omega\geq \max\{q,r\}$. The case $p_\omega\in[2,\max\{q,r\}]$ can now be proved by combining the fact that $L^a(\Omega)\hookrightarrow L^b(\Omega)$ for $1\leq b\leq a\leq \infty$ if $\Omega$ is a probability space. Finally, \eqref{eq:prob-pathwise-time} follows by combining Lemma \ref{2.4 lem}.
\end{proof}

\begin{remark}\label{can lebesgue}
From the proof of Lemma \ref{lem:prob-Besov} it is easy to verify that the estimates \eqref{eq:prob-Besov} and \eqref{eq:prob-pathwise-time} will continue to hold, when the space-time Besov space is replaced by either Lebesque or Soblev spaces (such as $L_{t}^q L_x^r$, $L_t^q H_x^{\alpha_x, r}$ etc.)
\end{remark}

\begin{remark}\label{vartheta explain}
For applications of Lemma \ref{lem:time-cutoff}, \ref{lem:prob-Besov} and also other random pathwise regularization results, the temporal exponent $\alpha$ will be chosen from line to line. By intersecting the good events, we may simply assume that there exist $0<\vartheta\ll 1$ and $\tilde{\vartheta}\in(0,1)$ such that the corresponding random estimates hold with a temporal prefactor $T^{\vartheta}$ on a good event with complement measure smaller than $\exp(-cT^{-\tilde \vartheta})$.
\end{remark}

\begin{remark}
For a given dyadic $N$, \eqref{eq:prob-Besov} also implies that
\[\|\psi\chi_{I_T}e^{it\Delta}P_N\phi^\omega\|_{L_\omega^{p_\omega}
 B^{\alpha_t}_{t,q,2}B^{\alpha_x}_{x,r,2}}
 \lesssim
 N^{\alpha_x}\sqrt{p_\omega}\,T^{1/q-\alpha_t}\|\phi\|_{H^{2\alpha_t}}\]
holds outside a set with measure $<\exp(-cT^{-\tilde\vartheta})$. For any $\vare>0$, by paying a cost of derivative loss $N^\vare$,  we may argue as in the proof of Lemma \ref{lem:block-event} below that
\[\|\psi\chi_{I_T}e^{it\Delta}P_N\phi^\omega\|_{L_\omega^{p_\omega}
 B^{\alpha_t}_{t,q,2}B^{\alpha_x}_{x,r,2}}
 \lesssim
 N^{\alpha_x+\vare}\sqrt{p_\omega}\,T^{1/q-\alpha_t}\|\phi\|_{H^{2\alpha_t}}\]  
holds uniformly in $N$ outside a set with measure $<\exp(-cT^{-\tilde\vartheta})$.
\end{remark}

\subsection{First randomized bilinear estimate}
The next counting lemma is the point at which the probabilistic argument
departs from the deterministic one.  We keep the coefficient frequency
$m\neq0$ fixed and count the Galilean block index $k$; this reverses the
order of counting used in \cite{KwakKwon} and produces the required
high-frequency gain.
\begin{lemma}\label{lem:k-slab-count}
Let $L,R,K>0$, $m\in\mathbb Z^d\setminus\{0\}$ and $\tau\in\mathbb R$.  Then
\begin{equation}\label{eq:k-slab-count}
 \#\left\{k\in\mathbb Z^d:|k|_\infty\sim K,
 |4Rk\cdot m-\tau|\leq L\right\}
 \lesssim K^{d-1}\left(1+\frac{L}{R|m|}\right).
\end{equation}
\end{lemma}

\begin{proof}
Choose $j$ such that $|m_j|\gtrsim |m|$ and Fix the remaining $d-1$ coordinates of $k$, which counts to $O(K^{d-1})$.  The condition in \eqref{eq:k-slab-count} confines $k_j$ to an interval of length at most $O(L/(R|m_j|))$.  Hence there are at most $C(1+L/(R|m|))$ possible values of $k_j$. Multiplication of the number of both choices yields the claim. 
\end{proof}

Next, for dyadic $N\geq32R$, let
\begin{equation*}
 C_{N,R,k}=\big(-2Rk+[-8R,8R]^d\big)
 \cap\{n:|n|_\infty\sim N\}.
\end{equation*}
By orthogonality, $C_{N,R,k}$ is not empty only if $k\sim N/R$. Set
\begin{equation*}
a_{N,R,k}^2=\sum_{n\in C_{N,R,k}}|b_n|^2.
\end{equation*}
Now recall the Galilean transform $I_\xi$ defined by \eqref{Galilean trans}. For $0<T<\frac12$, define
\begin{equation*}
 X_{N,R,k}
 =R^{-2\alpha}
 \|\psi\chi_{I_T}P_{\leq8R}I_{2Rk}e^{it\Delta}\phi_N^\omega
 \|_{B^\alpha_{t,p,2}L_x^p},
\end{equation*}
where $\alpha$ is the number defined in \eqref{eq:constant-dictionary-alpha}. We have the following random-type estimate for $X_{N,R,k}$.
\begin{lemma}\label{lem:block-event}
For every $p_\omega\geq2$ it holds
\begin{equation*}
 \|X_{N,R,k}\|_{L_\omega^{q_\omega}}
 \lesssim\sqrt{p_\omega}T^{1/p-\alpha}(R/N)^{\frac{d}{2}}N^{-s_c+\delta}.
\end{equation*}
Moreover, for any $\vare>0$, there exist $0<\vartheta\ll1$, $\tilde\vartheta\in(0,1)$, $C,c>0$ and $\Omega_T\subset \Omega$ satisfying $\mathbb P(\Omega_T^c)<\exp(-cT^{-\tilde\vartheta})$ such that for any $\omega\in\Omega_T$ it holds
\begin{equation}\label{eq:block-global}
X_{N,R,k}
 \leq C_\vare T^{\vartheta}N^\vare (R/N)^{\frac{d}{2}}N^{-s_c+\delta}
\end{equation}
uniformly in $N,R,k$ for which the set $C_{N,R,k}$ is not empty.
\end{lemma}

\begin{proof}
By direct calculation,
\[
 P_{\leq8R}I_{2Rk}e^{it\Delta}\phi_N^\omega
 =\sum_{n\in C_{N,R,k}}b_ng_n(\omega)
 e^{i(n+2Rk)\cdot x-it|n+2Rk|^2}.
\]
For $n\in C_{N,R,k}$, $|n+2Rk|\lesssim R$.  The Gaussian argument in
Lemma \ref{lem:prob-Besov} then yield
\begin{equation*}
 \|X_{N,R,k}\|_{L_\omega^{q_\omega}}
 \lesssim\sqrt{p_\omega}R^{-2\alpha}|n+2Rk|^{2\alpha}\,T^{1/p-\alpha}a_{N,R,k}
 \lesssim\sqrt{p_\omega}T^{1/p-\alpha}a_{N,R,k}.
\end{equation*}
Using Lemma \ref{lem:prob-Besov}, there exist $C,c>0$ such that for any $N,R,k$ such that $C_{N,R,k}\neq\varnothing$, there exists some $\Omega_{N,R,k}\subset\Omega$ satisfying $\mathbb P(\Omega_{N,R,k}^c)<\exp(-cN^{2\vare}T^{
2(\vartheta-(\frac{1}{p}-\alpha))})$ such that for any $\omega\in\Omega_{N,R,k}$ it holds
\begin{align*}
X_{N,R,k}\leq C T^{\vartheta}N^\vare a_{N,R,k}\lesssim C T^{\vartheta}N^\vare (R/N)^{\frac{d}{2}}N^{-s_c+\delta}
\end{align*}
by combining the fact that $a_{N,R,k}\lesssim (R/N)^{\frac{d}{2}}N^{-s_c+\delta}$. 

Finally, we upgrade the probabilistic estimate to one which holds uniformly in $N,R,k$. Set $\Omega_T:=\cap_{N,R,k:\,C_{N,R,k}\neq\varnothing}\Omega_{N,R,k}$. Then
\begin{equation*}
\begin{aligned}
\mathbb P(\Omega_T^c)\leq\sum_{N}\sum_{R\lesssim N}\sum_k P(\Omega_{N,R,k}^c)&\lesssim \sum_N \log N\cdot N^d 
\exp(-cN^{2\vare}T^{
2(\vartheta-(\frac{1}{p}-\alpha))})\\
&\lesssim \exp(-c'T^{
2(\vartheta-(\frac{1}{p}-\alpha))})
\end{aligned}
\end{equation*}
with some smaller $c'>0$, as desired. 
\end{proof}

Let us now recall the shear operator 
$$J_\xi A(t,x)=A(t,x-2t\xi)$$ 
defined in \cite{KwakKwon}.
\begin{lemma}\label{lem:weighted-shear}
Let $0<T<\frac12$, $N\geq32R$, $X_k:=X_{N,R,k}$ and $A_R:=P_RA$ satisfying $P_0 A=0$.  Then for any given $\vare>0$ there exist $0<\vartheta\ll 1$, $\tilde{\vartheta}\in(0,1)$ and $C,c>0$, such that for any $0<T\ll 1$ there exists some $\Omega_T\subset \Omega$ satisfying $\mathbb P(\Omega_T^c)<\exp(-cT^{-\tilde\vartheta})$, such that for any $\omega\in\Omega_T$ it holds
\begin{equation}\label{eq:weighted-shear}
\begin{aligned}
 \|X_k J_{2Rk}A_R\|_{\ell_k^2 B^{-1/8}_{t,2,2}L_x^2}&\leq C T^{\vartheta}N^{-s_c+\delta+\varepsilon}
 R^{d/2-1/4}\\
 &\times \left[
 \left(\frac RN\right)^{1/8}
 +R^{-1/4}\left(\frac{R^2}{N}\right)^{1/2}
 \right]
 \|A_R\|_{B^{1/2}_{t,1,1}L_x^1}
\end{aligned}
\end{equation}
uniformly in $N,R$.
\end{lemma}

\begin{proof}
Decompose $A_R=\sum_M P_M^tA_{R}$.  Taking the space-time Fourier transform we obtain
\[
 \mathcal{F}_{t,x}(J_{2Rk}P_M^t A_{R})(\eta,m)
 =\mathcal{F}_{t,x}(P_MA_{R})(\eta+4Rk\cdot m,m).
\]
Deducing similarly as in the proof of \cite[Lem. 3.7]{KwakKwon}, we obtain
\begin{equation*}
\begin{aligned}
 \|J_{2Rk}P_M^t A_{R}\|_{B^{-1/8}_{t,2,2}L_x^2}^2
 &\lesssim \|P_M^t A_R\|^2_{L_{t,x}^1}\sum_{|m|_\infty\sim R}\int_{|\eta|\lesssim M}
 \langle\eta+4Rk\cdot m\rangle^{-1/4}
d\eta.
\end{aligned}
\end{equation*}
The $m=0$ term is absent by assumption when $R=1$.  This in turn implies
\begin{equation*} 
\begin{aligned}
 &\sum_k \|J_{2Rk}P_M^t A_{R}\|_{B^{-1/8}_{t,2,2}L_x^2}^2\notag\\
 &\quad\lesssim\|P_M^t A_{R}\|_{L_{t,x}^1}^2
 \sum_{|m|_\infty\sim R}\int_{|\eta|\lesssim M}
 \sum_k \langle\eta-4Rk\cdot m\rangle^{-1/4}d\eta\\
 &\quad\lesssim \|P_M^t A_{R}\|_{L_{t,x}^1}^2
 \sum_{|m|_\infty\sim R}\int_{|\eta|\lesssim M}
\sum_{L\text{\,dyadic}}L^{-\frac14}\sum_{k:\,\langle\eta-4Rk\cdot m\rangle\sim L}d\eta.
\end{aligned}
\end{equation*}
By Lemma \ref{lem:k-slab-count}, for $L\leq NR$ we have
\[
 \sum_{k:|\eta-4Rk\cdot m|\sim L}1
 \lesssim (N/R)^d
 \left(\frac RN+\frac{L}{NR}\right).
\]
In the case $L>NR$, we simply use the global bound $\sum_{k:|\eta-4Rk\cdot m|\sim L}\lesssim (N/R)^{d}$. Combining the bound \eqref{eq:block-global}, and the fact that there are $O(R^d)$ spatial frequencies $m$ and the $\eta$ integral has length $O(M)$, we obtain
\begin{equation*}
\begin{aligned}
 &\sum_k X_k^2\|J_{2Rk}P_M^t A_{R}\|_{B^{-1/8}_{t,2,2}L_x^2}^2
 \\
 &\quad\lesssim T^{2\vartheta}N^{2(-s_c+\delta+\varepsilon)}R^d M\|P_M^t A_{R}\|_{L_{t,x}^1}^2\\
&\qquad\times \left[
 \frac RN\sum_{L\leq NR}L^{-1/4}
 +\frac1{NR}\sum_{L\leq NR}L^{3/4}
 +\sum_{L>NR}L^{-1/4}
 \right]\\
 &\quad\lesssim T^{2\vartheta}N^{2(-s_c+\delta+\varepsilon)}R^d M
 \left[\frac RN+(NR)^{-1/4}\right]\|P_M^t A_{R}\|_{L_{t,x}^1}^2.
\end{aligned}
\end{equation*}
The desired claim now follows from an $\ell_M^2$ summation.
\end{proof}

\begin{lemma}\label{lem:weighted-shear-interp}
Let the conditions in Lemma \ref{lem:weighted-shear} be retained. With $q_0,r_0,\sigma_1$ as in Section \ref{dictionary}, we have
\begin{equation}\label{eq:weighted-interp}
\begin{aligned}
 \|X_kJ_{2Rk}A_R\|_{\ell_k^2 B^{-\sigma_1}_{t,q_0,\infty}L_x^{(p/2)'}}
 &\leq C T^{\vartheta}N^{-s_c+\delta+\varepsilon}
 R^{d(1/r_0+2/p-1)-2\sigma_1}\\
&\times \left[
 \left(\frac RN\right)^{\sigma_1}
 +R^{-2\sigma_1}\left(\frac{R^2}{N}\right)^{4\sigma_1}
 \right]
\|A_R\|_{B^{\frac{1}{r_0}-\frac{1}{q_0}}_{t,r_0,r_0}L_x^{r_0}}
\end{aligned}
\end{equation}
uniformly in $N,R$.
\end{lemma}

\begin{proof}
First recall the trivial bound 
\begin{align}\label{interpolation}
\|J_{2Rk}A_R\|_{B^{0}_{t,q,\infty}L_x^q}
\lesssim
\|J_{2Rk}A_R\|_{L^q_{t,x}}
=
\|A_R\|_{L^q_{t,x}}
\lesssim
R^{d\left(\frac1r-\frac1q\right)}
\|A\|_{B^{\frac1r-\frac1q}_{t,r,r}L_x^r}.
\end{align}
This implies 
\begin{equation}\label{inter2}
\begin{aligned}
 \|X_kJ_{2Rk}P_RA\|_{\ell_k^2B^0_{t,q,\infty}L_x^q}
 &\lesssim \left(\sum_kX_k^2\right)^{1/2}
 R^{d(1/r-1/q)}
 \|A_R\|_{B^{1/r-1/q}_{t,r,r}L_x^r}\\
 &\lesssim T^{\vartheta}N^{-s_c+\delta+\vare}R^{d(1/r-1/q)}
 \|A_R\|_{B^{1/r-1/q}_{t,r,r}L_x^r}.
\end{aligned}
\end{equation}
A complex interpolation between \eqref{eq:weighted-shear} and \eqref{inter2} yields \eqref{eq:weighted-interp}.
\end{proof}

\begin{lemma}\label{thm:random-KK}
Let the conditions in Lemma \ref{lem:weighted-shear} be retained and let $\theta$ be the number defined in \eqref{eq:constant-dictionary-KK}. Then for function $v$ defined on $\R\times\T^d$ we have
\begin{equation}\label{eq:random-KK}
\begin{aligned}
 &\left|\int\psi^2\chi_{I_T}e^{it\Delta}\phi_N^\omega \overline{v_M}A_R\,dxdt\right|&\\
 &\quad\lesssim T^{\vartheta}N^{-s_c+\delta+\varepsilon}\|v_M\|_{Z^0}
 R^{-\sigma}\left(
 \left(\frac RN\right)^{\sigma_1}
 +R^{-2\sigma_1}\min\{1,(R^2/N)^{4\sigma_1}\}
 \right)
 R^\theta\|A_R\|_{B^{\frac{1}{r_0}-\frac{1}{q_0}}_{t,r_0,r_0}L_x^{r_0}}
\end{aligned}
\end{equation}
uniformly in $N,M,R$.
\end{lemma}

\begin{proof}
Arguing as in the proof of \cite[Lem. 3.9]{KwakKwon}, we arrive at 
\begin{align*}
\text{l.h.s. of \eqref{eq:random-KK}}\lesssim\sum_k
R^{4\alpha} X_kY_k
 \|J_{2Rk}A_R\|_{B^{-\sigma_1}_{q_0,\infty}L^{(p/2)'}},
\end{align*}
where 
$$Y_k= R^{-2\alpha}\|\psi P_{\leq8R}I_{2Rk}e^{it\Delta}v_M\|_{B^\alpha_{p,2}L^p}.$$ 
By the definition of $Z^0$-norm, it holds $\|R^{-\sigma}Y_k\|_{\ell_k^2}\sim\|v_M\|_{Z^0}$. 

In the case $R\leq N^{\frac12}$, Lemma \ref{lem:weighted-shear-interp} and Cauchy-Schwarz in $k$ immediately yield
\begin{align*}
&\lesssim T^{\vartheta}N^{-s_c+\delta+\varepsilon}\|v_M\|_{Z^0}
 \left[\left(\frac RN\right)^{\sigma_1}
 +R^{-2\sigma_1}\left(\frac{R^2}{N}\right)^{4\sigma_1}\right]\\
&\qquad\times R^{4\alpha+\sigma+d(1/r_0+2/p-1)-2\sigma_1}\|A_R\|_{B^{\frac{1}{r_0}-\frac{1}{q_0}}_{t,r_0,r_0}L_x^{r_0}},
\end{align*}
The claim follows by combining the scaling and constant identities
\[
 d\left(\frac1{r_0}+\frac2p-1\right)-2\sigma_1
 =\theta-2\beta\qquad\text{and}\qquad 4\alpha+\sigma+\theta-2\beta=\theta-\sigma.
\]
The case $R>N^{\frac12}$ can be similarly dealt, as long as we replace \eqref{eq:weighted-interp} by the bilinear estimate 
\[\text{l.h.s. of \eqref{eq:random-KK}}
\lesssim
\sum_k R^{4\alpha}X_kY_k
\left(
\langle N/R\rangle^{-\sigma_1}
+
R^{-2\sigma_1}
\right)
R^\theta
\|A_R\|_{B^{\frac1{r_0}-\frac1{q_0}}_{t,r_0,r_0}L_x^{r_0}}\]
given in the proof of \cite[Lem. 3.9]{KwakKwon}. This completes the proof.
\end{proof}

\begin{corollary}\label{cor:random-KK-summed}
Let the conditions in Lemma \ref{lem:weighted-shear} be retained. Then
\begin{equation}\label{eq:random-KK-summed-form}
 \left|\int\psi^2\chi_{I_T}e^{it\Delta}\phi_N^\omega \overline{v_M} P_{\leq N/32}A\,dxdt\right|
 \lesssim T^{\vartheta}N^{-s_c+\delta-\sigma+\varepsilon}\|v_M\|_{Z^0}\|A\|_{B^{\frac{1}{r_0}-\frac{1}{q_0}}_{t,r_0,r_0}
B^\theta_{x,r_0,\infty}}
\end{equation}
uniformly in $N,M$.
\end{corollary}

\begin{proof}
This follows immediately from Lemma \ref{thm:random-KK}, combining also the triangular inequality applied for $P_{\leq N/32}A=\sum_{R:\,R\leq N/32}A_R$, the uniform in $R$ estimate 
$$R^\theta\|A_R\|_{B^{\frac{1}{r_0}-\frac{1}{q_0}}_{t,r_0,r_0}L_x^{r_0}}\lesssim 
\|A\|_{B^{\frac{1}{r_0}-\frac{1}{q_0}}_{t,r_0,r_0}
B^\theta_{x,r_0,\infty}}$$ 
and the upper bound for the dyadic sum 
$$\sum_{R:R\leq N/32} R^{-\sigma}\left(
 \left(\frac RN\right)^{\sigma_1}
 +R^{-2\sigma_1}\min\{1,(R^2/N)^{4\sigma_1}\}
 \right)\lesssim N^{-\sigma}+N^{-\sigma_1-\sigma/2}\lesssim N^{-\sigma},$$
where we also used the fact that $\sigma_1\gg \sigma$.
\end{proof}

\begin{lemma}\label{first bilinear}
Let $A$ satisfy $P_0A=0$. Then for any given $\vare>0$ there exist $0<\vartheta\ll 1$, $\tilde{\vartheta}\in(0,1)$ and $C,c>0$, such that for any $0<T\ll 1$ there exists some $\Omega_T\subset \Omega$ satisfying $\mathbb P(\Omega_T^c)<\exp(-cT^{-\tilde\vartheta})$, such that for any $\omega\in\Omega_T$ it holds
\begin{align*}
\left|
\int_{\mathbb{R}\times\mathbb{T}^d}
\psi^2 (\chi_{I_T}e^{it\Delta}\phi^\omega)\overline{v} A
\,dx\,dt
\right|
\lesssim
T^{\vartheta}\|\phi\|_{H_x^{\vare-\sigma}}\|v\|_{Z^0}
\|A\|_{B^{\frac{1}{r_0}-\frac{1}{q_0}}_{t,r_0,r_0}B^\theta_{x,r_0,\infty}}.
\end{align*}
\end{lemma}

\begin{proof}
We decompose the integral to 
\begin{align*}
&\le
\left|
\int_{\mathbb{R}\times\mathbb{T}^d}
\psi^2 \chi_{I_T}e^{it\Delta}\phi^\omega\cdot\pi_{\le}\!\left(\overline{v},A\right)
\,dx\,dt
\right|\\
&\qquad+
\left|
\int_{\mathbb{R}\times\mathbb{T}^d}
\psi^2 \chi_{I_T}e^{it\Delta}\phi^\omega\cdot\pi_{>}\!\left(\overline{v},A\right)
\,dx\,dt
\right|,\end{align*}
where $\pi_{\leq}$ and $\pi_{>}$ are the paraproduct operators defined in \eqref{paraproducts}. For the first part, we use \eqref{2.14},  \eqref{eq:prob-pathwise-time} and \eqref{3.15} to obtain
\begin{align*}
&\left|\int_{\mathbb{R}\times\mathbb{T}^d}
\psi^2 \chi_{I_T}e^{it\Delta}\phi^\omega\cdot\pi_{\le}\!\left(\overline{v},A\right)
\,dx\,dt\right|\\
&\quad\lesssim \|\psi\chi_{I_T}e^{it\Delta}\phi^\omega\|_{B^{\widetilde{\alpha}}_{t,p,2}B^{-\widetilde{\beta}}_{x,p,2}}
\|v\|_{B^{\widetilde{\alpha}}_{t,p,2}B^{-\widetilde{\beta}}_{x,p,2}}
\|A\|_{B^{\frac{1}{r_0}-\frac{1}{q_0}}_{t,r_0,r_0}B^\theta_{x,r_0,\infty}}\\
&\quad\lesssim T^{\vartheta}\|\phi\|_{H_x^{-\sigma}}\|v\|_{Z^0}
\|A\|_{B^{\frac{1}{r_0}-\frac{1}{q_0}}_{t,r_0,r_0}B^\theta_{x,r_0,\infty}},
\end{align*}
where $\tilde\alpha,\tilde\beta$ are the constants defined in \eqref{eq:constant-dictionary-alpha}. For the second part, we use \eqref{eq:random-KK-summed-form} to deduce
\begin{align*}
&\left|\int_{\mathbb{R}\times\mathbb{T}^d}
\psi^2 \chi_{I_T}e^{it\Delta}\phi^\omega\cdot\pi_{>}\!\left(\overline{v},A\right)
\,dx\,dt\right|
\lesssim\sum_{N\sim M}\left|\int\psi^2\chi_{I_T}e^{it\Delta}\phi_N^\omega \overline{v_M} P_{\leq N/32}A\,dxdt\right|\\
&\quad \lesssim T^{\vartheta}\left(\sum_N N^{-2(s_c+\sigma-\delta-\varepsilon)}\right)^{\frac12}\|v\|_{Z^0}\|A\|_{B^{\frac{1}{r_0}-\frac{1}{q_0}}_{t,r_0,r_0}
B^\theta_{x,r_0,\infty}} \\
&\quad\lesssim
 T^{\vartheta}\|\phi\|_{H_x^{\vare-\sigma}}\|v\|_{Z^0}\|A\|_{B^{\frac{1}{r_0}-\frac{1}{q_0}}_{t,r_0,r_0}
B^\theta_{x,r_0,\infty}},
\end{align*} 
from which the desired claim follows.
\end{proof}

\subsection{Second randomized bilinear estimate}\label{sec:opposite}
When the integral is involved with $\overline{u}$ instead of $u$, the block decomposition should be written to 
\begin{align*}
 \int_{\mathbb R\times\mathbb T^d}\psi^2\chi_{I_T}\overline u\,\overline v\,B\,dxdt
 =\sum_{k\in \Z^d}\int_{\mathbb R\times\mathbb T^d}
 \psi^2\chi_{I_T}\overline{P_{C_k} u}\,
 \overline{P_{Q_k} v}\, B\,dxdt,
\end{align*}
where 
\begin{align*}
C_k:=2Rk+[-8R,8R]^d,\qquad Q_k:=-2Rk+(-R,R]^d.
\end{align*}
Motivated by such a frequency decomposition, for
$\xi\in\mathbb Z^d$ we define the renormalized Galilean transform and the
opposite-phase shear by
\begin{align}
 I_\xi^{\mathrm{op}} u(t,x)
 &:=e^{4it|\xi|^2}I_\xi u(t,x),
 \label{eq:opposite-I}\\
 J_\xi^{\mathrm{op}} B(t,x)
 &:=e^{2ix\cdot\xi+2it|\xi|^2}B(t,x-2t\xi). \notag
\end{align}
A direct change of variables and spatial frequency decomposition gives
\begin{equation}\label{eq:opposite-Galilean-identity}
 \int_{\mathbb R\times\mathbb T^d}\psi^2 \chi_{I_T}\overline u\,\overline v\,B\,dxdt
 =\sum_{k\in\Z^d}\int_{\mathbb R\times\mathbb T^d}
 \overline{I_{2Rk}^{\mathrm{op}}\psi \chi_{I_T} P_{C_k} u}\,
 \overline{I_{2Rk} \psi \chi_{I_T} P_{Q_k} v}\,J_{2Rk}^{\mathrm{op}} B\,dxdt.
\end{equation}
The renormalization in \eqref{eq:opposite-I} is chosen so
that
\begin{equation}\label{eq:opposite-renormalized-phase}
 I_\xi^{\mathrm{op}}
 \big(e^{in\cdot x-it|n|^2}\big)
 =e^{i(n+\xi)\cdot x-it\left(|n+\xi|^2-4|\xi|^2\right)}.
\end{equation}
In particular, $\left||n+\xi|^2-4|\xi|^2\right|\lesssim NR$ holds if $|\xi|\sim N$ and $|n-\xi|\lesssim R$. Finally, define
\begin{align*}
X_{N,R,k}^{\mathrm{op}}:=R^{-3\sigma_1}\|I^{\mathrm{op}}_{2Rk}\psi\chi_{I_T}P_{C_k}e^{it\Delta}\phi_N^\omega
 \|_{B^{\frac{3\sigma_1}{2}}_{t,p,2}L_x^p}.
\end{align*}

\begin{lemma}\label{lem:opposite-shear-endpoint}
Let $0<T<\frac12$, $N\geq32R$, $X_k:=X^{\mathrm{op}}_{N,R,k}$ and $B_R:=P_RB$.  Then for any given $\vare>0$ there exist $0<\vartheta\ll 1$, $\tilde{\vartheta}\in(0,1)$ and $C,c>0$, such that for any $0<T\ll 1$ there exists some $\Omega_T\subset \Omega$ satisfying $\mathbb P(\Omega_T^c)<\exp(-cT^{-\tilde\vartheta})$, such that for any $\omega\in\Omega_T$ it holds
\begin{equation}\label{eq:opposite-shear-endpoint}
 \|X_k J_{2Rk}^{\mathrm{op}} B_R\|_{\ell_k^2 B^{-1/8}_{t,2,2}L_x^2}
 \leq CT^{\vartheta}N^{-s_c+\delta+\vare} 
R^{\frac{d}{2}-\frac14}(N/R)^{-\frac14+\frac{3\sigma_1}{2}}
 \|B_R\|_{B^{1/2}_{t,1,1}L_x^1}
\end{equation}
uniformly in $N,R$.
\end{lemma}

\begin{proof}
Using \eqref{eq:opposite-renormalized-phase} and the fact that $\left||n+\xi|^2-4|\xi|^2\right|\lesssim NR$ holds if $|\xi|\sim N$ and $|n-\xi|\lesssim R$, we may argue as in the proof of Lemma \ref{lem:block-event} to conclude that on a good event, 
\begin{align}\label{3.27}
\|X_k\|_{\ell_k^2}\leq C_\vare T^{\vartheta}N^{-s_c+\delta+\vare}(N/R)^{\frac{3\sigma_1}{2}}.
\end{align}
Next we prove that for fixed $N,R$, it holds
\begin{align}\label{3.28}
\|J_{2Rk}^{\mathrm{op}} B_R\|_{B^{-1/8}_{t,2,2}L_x^2}\lesssim R^{\frac{d}{2}-\frac14}(N/R)^{-\frac14}
 \|B_R\|_{B^{1/2}_{t,1,1}L_x^1}
 \end{align}
uniformly in $k$. First notice that for $\xi\in\Z^d$ we have
$$ \mathcal{F}_{t,x}(J_{2Rk}^{\mathrm{op}}P_M^t B_{R})(\eta,m+4Rk)
 =\mathcal{F}_{t,x}B_{R}
 \big(\eta+4Rk\cdot m-8R^2|k|^2,m\big). $$
Then arguing as in the proof of Lemma \ref{lem:weighted-shear}, we arrive at
\begin{equation*}
\begin{aligned}
\|J_{2Rk}^{\mathrm{op}}P_M^t B_{R}\|_{B^{-1/8}_{t,2,2}L_x^2}^2
 &\lesssim \|P_M^t B_R\|^2_{L_{t,x}^1}\sum_{|m|_\infty\sim R}\int_{|\eta|\lesssim M}
 \langle\eta+4Rk\cdot m-8R^2|k|^2\rangle^{-1/4}
d\eta.
\end{aligned}
\end{equation*}
Using $|k|\sim N/R$ we have 
$$\langle\eta+4Rk\cdot m-8R^2|k|^2\rangle\gtrsim N^2.$$ 
We hence discuss $M\ll N^2$ and $M\gtrsim N^2$ separately. In the former case we have the rough estimate
$$ \lesssim  \|P_M^t B_R\|^2_{L_{t,x}^1} R^{d}M N^{-\frac12}\sim 
\|P_M^t B_R\|^2_{L_{t,x}^1} R^{d-\frac12}M (N/R)^{-\frac12}.$$
In the latter case, we use instead the estimate
$$\int_{|\eta|\lesssim M}
 \langle\eta+4Rk\cdot m-8R^2|k|^2\rangle^{-1/4}
d\eta\lesssim 1+M^\frac34\lesssim M N^{-\frac12}.$$
\eqref{3.28} follows now from an $\ell_M^2$ summation. The final claim follows then from first taking $\ell_k^\infty$-norm to \eqref{3.28} and then $\ell_k^2$-norm to $(X_k)_k$.
\end{proof}

\begin{lemma}\label{lem:weighted-shear-interp2}
Let the conditions in Lemma \ref{lem:opposite-shear-endpoint} be retained. With $q_0,r_0,\sigma_1$ as in Section \ref{dictionary}, we have
\begin{equation*}
\begin{aligned}
 \|X_kJ_{2Rk}^{\rm op}B_R\|_{\ell_k^2 B^{-\sigma_1}_{t,q_0,\infty}L_x^{(p/2)'}}
 &\leq C T^{\vartheta}N^{-s_c+\delta+\varepsilon}
 R^{d(1/r_0+2/p-1)-2\sigma_1}
 \left(\frac RN\right)^{\frac{\sigma_1}{2}}
\|B_R\|_{B^{\frac{1}{r_0}-\frac{1}{q_0}}_{t,r_0,r_0}L_x^{r_0}}
\end{aligned}
\end{equation*}
uniformly in $N,R$.
\end{lemma}

\begin{proof}
By using the interpolation as written in the proof of Lemma \ref{lem:weighted-shear-interp}, the claim follows from \eqref{3.27}, \eqref{eq:opposite-shear-endpoint} and \eqref{interpolation}. Notice that the exponent $\sigma_1/2$ is calculated as 
$$-\frac{\sigma_1}{2}=8\sigma_1\cdot(-\frac14+\frac{3\sigma_1}{2})+(1-8\sigma_1)\cdot \frac{3\sigma_1}{2}.$$
\end{proof}

\begin{lemma}\label{second bilinear}
For any given $\vare>0$ there exist $0<\vartheta\ll 1$, $\tilde{\vartheta}\in(0,1)$ and $C,c>0$, such that for any $0<T\ll 1$ there exists some $\Omega_T\subset \Omega$ satisfying $\mathbb P(\Omega_T^c)<\exp(-cT^{-\tilde\vartheta})$, such that for any $\omega\in\Omega_T$ it holds
\begin{align*}
\left|
\int_{\mathbb{R}\times\mathbb{T}^d}
\psi^2 \overline{(\chi_{I_T}e^{it\Delta}\phi^\omega) v}B
\,dx\,dt
\right|
\lesssim
T^{\vartheta}\|\phi\|_{H_x^{\vare-\sigma}}\|v\|_{Z^0}
\|B\|_{B^{\frac{1}{r_0}-\frac{1}{q_0}}_{t,r_0,r_0}B^\theta_{x,r_0,\infty}}.
\end{align*}
\end{lemma}

\begin{proof}
The proof is almost identical to the one of Lemma \ref{first bilinear}, by correspondingly applying Lemma \ref{lem:weighted-shear-interp2} and \eqref{eq:opposite-Galilean-identity} therein. The only obstacle is that in the opposite-phase case, Lemma \ref{kwak_lem_lhh} is no longer valid if we insist on the constant exponents applied in the proof of Lemma \ref{first bilinear}. Indeed, we still need to check that Lemma \ref{kwak_lem_lhh} will continue to hold by choosing the new exponents. We then define $\alpha_{\mathrm op}$ as in \eqref{2.29}. Direct calculation shows that
$$\alpha_{\mathrm op}=
\frac{\sigma_3-2\sigma}{d+2}
-\frac12\sigma_1.$$ 
Setting $s_1=\frac{3\sigma_1}{2}$, $s_2=\alpha_{\mathrm op}$ and $s_3=-\sigma_1$ we see that the conditions of Lemma \ref{kwak_lem_lhh} are satisfied, as desired.
\end{proof}

\section{Nonlinear estimates}\label{nonlin_sec}
We now insert the two randomized bilinear estimates into an exact Bony
linearization of the gauged nonlinearity.  The decomposition below separates
the nonlinear increment into a mean-free same-phase term, an opposite-phase
term, and a scalar remainder.  The first two are controlled by the estimates
of Section \ref{sec 333}, whereas the scalar remainder requires a separate shellwise
argument.

\subsection{Bony decomposition adapted to the gauged nonlinearity}
Let
\[
 U=y_{\leq N/2},
 \qquad h=y_N,
 \qquad W_\vartheta=U+\vartheta h.
\]
The Wirtinger derivatives are
\begin{equation*}
 \partial_zF(z)=\left(1+\frac{a}{2}\right)|z|^a,
 \qquad
 \partial_{\bar z}F(z)=\frac a2|z|^{a-2}z^2,
\end{equation*}
where the second expression is defined to be zero at $z=0$.  Hence
\begin{equation}\label{eq:F-bony}
 F(U+h)-F(U)=hA(U,h)+\overline hB(U,h),
\end{equation}
with
\begin{align*}
 A(U,h)&=c_a\int_0^1|W_\vartheta|^a\,d\vartheta,
 \\
 B(U,h)&=\frac a2\int_0^1|W_\vartheta|^{a-2}W_\vartheta^2\,d\vartheta.
\end{align*}
This is the exact Bony decomposition formula applied in 
\cite{KwakKwon}. In the following we shall establish a modified Bony decomposition adapted to the gauged nonlinearity.

\begin{lemma}\label{lem:gauged-increment}
We have the following Bony decomposition adapted to the gauged nonlinearity $\mathcal G$ defined in \eqref{eq:ca-def}:
\begin{equation}\label{eq:G-increment}
 G(U+h)- G(U)
 =hA_0(U,h)+\overline hB(U,h)+R(U,h),
\end{equation}
where
\begin{equation*}
 A_0(U,h)=A(U,h)-\fint_{\mathbb T^d}A(U,h)\,dx
\end{equation*}
has zero spatial mean and
\begin{align}
 R(U,h)
 &=r_0(U,h)h-c_ar_1(U,h)U, \notag
 \\
 r_0(U,h)
 &=\fint A(U,h)\,dx-c_a\mu(U+h), \notag
 \\
 r_1(U,h)&=\mu(U+h)-\mu(U),
 \label{eq:r1-def}
\end{align}
where $c_a$ and $\mu(y)$ are defined by \eqref{eq:ca-def}. With $\alpha_a=\min\{1,a\}$, it also holds
\begin{equation}\label{eq:r-bounds}
 |r_0(U,h)|+|r_1(U,h)|
 \lesssim
 \fint |h|^{\alpha_a}(|U|+|h|)^{a-\alpha_a}\,dx.
\end{equation}
\end{lemma}

\begin{proof}
By \eqref{eq:F-bony},
\begin{align*}
 \mathcal G(U+h)-\mathcal G(U)
 &=hA+\overline hB-c_a\mu(U+h)h
   -c_a\big(\mu(U+h)-\mu(U)\big)U.
\end{align*}
Add and subtract $h\fint A\,dx$ to obtain \eqref{eq:G-increment}.

For $x,y\in\mathbb C$ and $a>0$ we have the elementary inequality
\begin{equation}\label{eq:power-increment}
 \big||x+y|^a-|x|^a\big|
 \lesssim |y|^{\alpha_a}(|x|+|y|)^{a-\alpha_a}.
\end{equation}
Formula
\eqref{eq:r1-def} and \eqref{eq:power-increment} give the bound for $r_1$.
Since
\[
 r_0=c_a\int_0^1\big(\mu(U+\vartheta h)-\mu(U+h)\big)d\vartheta,
\]
applying \eqref{eq:power-increment} with
$x=U+\vartheta h$ and $y=(1-\vartheta)h$, then integrating in $\vartheta$,
gives the bound for $r_0$.
\end{proof}

Grouping the terms, set
\begin{align*}
G_1^N:=hA_0(U,h)+\overline hB(U,h),\quad
G_2^N:=R(U,h).
\end{align*}
We shall estimate $\sum_N G_2^N$ directly and establish estimates for $G_{1,K}^N:=P_KG_1^N$ according to the size of the dyadic number $K$. More precisely, we establish frequency localized estimates in  
\begin{itemize}
\item The low-high region: $K\leq 4N$, and
\item The high-low region: $K\geq 4N$.
\end{itemize} 

Before proceeding, we first record a useful mixed random-deterministic estimate for the nonlinear potential $|u|^{k-m} u^m$ with $k\in\{a/2,a\}$ and $m\in\Z$. 

\begin{lemma}\label{random zs norm_lem}
For given $m\in\Z$ there exist $0<\vartheta\ll 1$, $\tilde{\vartheta}\in(0,1)$ and $C,c>0$, such that for any $0<T\ll 1$ there exists some $\Omega_T\subset \Omega$ satisfying $\mathbb P(\Omega_T^c)<\exp(-cT^{-\tilde\vartheta})$, such that for any $\omega\in\Omega_T$ it holds
\begin{align*}
\begin{aligned}
 \left\|\psi\chi_{I_T}\left(|e^{it\Delta}\phi^\omega+u|^{a-m}(e^{it\Delta}\phi^\omega+u)^m\right)\right\|_{
 B^{\frac{1}{r_0}-\frac{1}{q_0}}_{t,r_0,r_0}
 B^{\theta}_{x,r_0,r_0}}
  \lesssim(T^{\vartheta}\|\phi\|_{H_x^{2s_0+s_1}}+\|u\|_{Z^{s_c}})^a,\\
 \left\|\psi\chi_{I_T}\left(|e^{it\Delta}\phi^\omega+u|^{\frac{a}{2}-m}(e^{it\Delta}\phi^\omega+u)^m\right)\right\|_{
 B^{\zeta}_{t,\widehat r,\widehat r}
 B^{\eta}_{x,\widehat r,\widehat r}}
  \lesssim(T^{\vartheta}\|\phi\|_{H_x^{2s_0+s_1}}+\|u\|_{Z^{s_c}})^{\frac{a}{2}}.
\end{aligned}
\end{align*}
\end{lemma}

\begin{proof}
This follows immediately from Lemma \ref{lem:mixed-coeff} and \ref{lem:prob-Besov}.
\end{proof}

Finally, we also recall that $y$ is composed of 
$$ y=e^{it\Delta}\phi^\omega+w$$
with $w\in Z^{s^*}$ {\it{a priori}}.

\subsection{Low-high estimation for $G_{1,K}^N$}
\begin{lemma}\label{final1}
Fix $\vare>0$ as in \eqref{eq:delta-epsilon-varepsilon-margins}. There exist $0<\vartheta\ll 1$, $\tilde{\vartheta}\in(0,1)$ and $C,c>0$, such that for any $0<T\ll 1$ there exists some $\Omega_T\subset \Omega$ satisfying $\mathbb P(\Omega_T^c)<\exp(-cT^{-\tilde\vartheta})$, such that for any $\omega\in\Omega_T$ it holds
\begin{equation*}
\begin{aligned}
&\|K^+(\psi_1\chi_{I_T}G_{1,K}^N)\|_{Y^{s^*}}\lesssim \|\psi_1\chi_{I_T}G_{1,K}^N\|_{(Z^{-s^*})'}\\
&\quad\lesssim
(T^{\vartheta}\|\phi\|_{H_x^{2s_0+s_1}}+\|w\|_{Z^{s^*}})^a
\left(\frac{K}{N}\right)^{s^*}
\left(T^{\vartheta}N^{-\sigma_\delta}+\|w_N\|_{Z^{s^*}}\right).
\end{aligned}
\end{equation*}
\end{lemma}

\begin{proof}
This follows from the proof of \cite[Lem. 4.1]{KwakKwon}, combined with
Lemmas \ref{random zs norm_lem}, \ref{first bilinear},
\ref{second bilinear}, and \eqref{3.11}.  The only additional
bookkeeping is
\[
 K^{\vare_0}\left(\frac KN\right)^{s_c}
 N^{-\sigma+\delta+2\vare}
 =\left(\frac KN\right)^{s^*}N^{-\sigma_\delta},
 \qquad
 K^{\vare_0}\left(\frac KN\right)^{s_c}N^{-\vare_0}
 =\left(\frac KN\right)^{s^*}.
\]
\end{proof}

\subsection{High-low estimation for $G_{1,K}^N$}
\begin{lemma}\label{final2}
Let $\nu^*,\mu^*$ be defined by \eqref{eq:star-parameters}, and fix $\vare>0$ as in \eqref{eq:delta-epsilon-varepsilon-margins}. There exist $0<\vartheta\ll 1$, $\tilde{\vartheta}\in(0,1)$ and $C,c>0$, such that for any $0<T\ll 1$ there exists some $\Omega_T\subset \Omega$ satisfying $\mathbb P(\Omega_T^c)<\exp(-cT^{-\tilde\vartheta})$, such that for any $\omega\in\Omega_T$ and $K\geq 4N$ it holds
\begin{equation}\label{final2+}
\begin{aligned}
&\|K^+(\psi_1\chi_{I_T}G_{1,K}^N)\|_{Y^{s^*}}\lesssim \|\psi_1\chi_{I_T}G_{1,K}^N\|_{(Z^{-s^*})'}\\
&\quad\lesssim
(T^{\vartheta}\|\phi\|_{L_x^2}+\|w\|_{Z^{s^*}})^a
\left(\frac{N}{K}\right)^{\mu^*}
\left(\sum_{L:L\leq N}\left(\frac{L}{N}\right)^{2\nu^*}
\left(T^{2\vartheta} L^{-2\sigma_\delta}+\|w_L\|^2_{Z^{s^*}}\right)\right)^{\frac12} .
\end{aligned}
\end{equation}
\end{lemma}

\begin{proof}
We recall the following formula proved in \cite[(4.6)]{KwakKwon}:
\[
 \|\psi_1G_{1,K}^N\|_{L_{t,x}^{p'}}
 \lesssim
 (N/K)^\mu K^{-s_c-\sigma}N^{-\nu}
 \|y_{\le N}\|_{L_t^{q_h}H_x^{s_c-\sigma+\nu,r_h}}
 \|y\|_{L_t^{q_h}H_x^{s_c-\sigma,r_h}}^a.
\]
Here $q_h,r_h$ are defined by \eqref{eq:qh-rh}.  Using
\eqref{eq:prob-Besov}, \eqref{3.14}, and
$Z^{s^*}\hookrightarrow Z^{s_c}$, we obtain
\[
 \|y\|_{L_t^{q_h}H_x^{s_c-\sigma,r_h}}^a
 \lesssim
 (T^\vartheta\|\phi\|_{L_x^2}+\|w\|_{Z^{s^*}})^a.
\]
$\|y_{\le N}\|_{L_t^{q_h}H_x^{s_c-\sigma+\nu,r_h}}$ can be estimated using Littlewood–Paley square-function estimate and Minkowski as follows:
\[
 \|y_{\le N}\|_{L_t^{q_h}H_x^{s_c-\sigma+\nu,r_h}}
 \lesssim
 \left(\sum_{L\le N}L^{2(s_c-\sigma+\nu)}
 \|y_L\|_{L_t^{q_h}L_x^{r_h}}^2\right)^{1/2}.
\]
Now \eqref{final2+} follows by combining
\[
 L^{s_c-\sigma+\nu}\|y_L\|_{L_t^{q_h}L_x^{r_h}}
 \lesssim
 T^\vartheta L^{\nu-\sigma+\delta+\vare}
 +L^{\nu-\vare_0}\|w_L\|_{Z^{s^*}},
\]
with
\[
 K^{s^*+\sigma}\left(\frac NK\right)^\mu
 K^{-s_c-\sigma}N^{-\nu}
 =\left(\frac NK\right)^{\mu^*}N^{-\nu^*},
\]
and the identities
\[
 N^{-\nu^*}L^{\nu-\sigma+\delta+\vare}
 =\left(\frac LN\right)^{\nu^*}
 L^{-(\sigma-\delta-\vare_0-\vare)}
 \leq
 \left(\frac LN\right)^{\nu^*}L^{-\sigma_\delta},
 \qquad
 N^{-\nu^*}L^{\nu-\vare_0}
 =\left(\frac LN\right)^{\nu^*}.
\]
Here we also used
$\|w_L\|_{L_t^{q_h}H_x^{s_c-\sigma+\nu,r_h}}
\lesssim L^{\nu-\vare_0}\|w_L\|_{Z^{s^*}}$ and
\[
 \|\psi_1G_{1,K}^N\|_{(Z^{-s^*})'}
 \lesssim K^{s^*+\sigma}
 \|\psi_1G_{1,K}^N\|_{L_{t,x}^{p'}},
\]
which follows from \eqref{3.14} and duality.
\end{proof}

\subsection{Estimation for $G_{2}^N$}
\begin{lemma}\label{lem:gauge-remainder}
For $s^*=s_c+\vare_0$, there exist $0<\vartheta\ll 1$, $\tilde{\vartheta}\in(0,1)$ and $C,c,c_1,c_2>0$, such that for any $0<T\ll 1$ there exists some $\Omega_T\subset \Omega$ satisfying $\mathbb P(\Omega_T^c)<\exp(-cT^{-\tilde\vartheta})$, such that for any $\omega\in\Omega_T$ it holds
\begin{align}\label{eq:gauge-rem-est}
\begin{aligned}
 \|K^+(\psi\chi_{I_T}\sum_N G_{2}^N)\|_{Y^{s^*}}&\lesssim \left\|\psi\chi_{I_T}\sum_N G_{2}^N\right\|_{L_t^1H_x^{s^*}}\\
&\lesssim
T^{\vartheta}\left(T^{\vartheta}\|\phi\|_{L_x^2}
 +\|w\|_{Z^{s^*}}\right)^{a-\alpha_a}
 \big(1+\|w\|_{Y^{s^*}(I)}\big)^{1+\alpha_a}.
\end{aligned}
\end{align}
\end{lemma}

\begin{proof}
Recall that $\alpha_a=\min\{1,a\}$. Let $r_{\text{rem}},q_{\text{rem}}$ be defined through \eqref{p>2} and \eqref{eq:qg-choice}. For every dyadic $N$, define
\begin{equation}\label{eq:scalar-bN-aN}
 {b_N=N^{s^*}\|\chi_{I_T}w_N\|_{L_t^\infty L_x^2},
 \qquad
 a_N=N^{-s_c+\delta+\vare}+N^{-s^*}b_N.}
\end{equation}
Also set
\begin{equation*}
 A_I=
T^{\vartheta}\|\phi\|_{L_x^2}
 +\|w\|_{Z^{s^*}}.
\end{equation*}
By definition, $ \|b_N\|_{\ell_N^2}\lesssim\|w\|_{Y^{s^*}(I)}$. Moreover, using Lemma \ref{lem:prob-Besov}, the condition $ \frac d{r_{\text{rem}}}+\frac2{q_{\text{rem}}}>\frac2a$ from \eqref{eq:qg-choice} and \eqref{3.14}, and also \eqref{3.14} and Sobolev embedding, it follows
\begin{align} \label{eq:scalar-yN-qg}
 \|\chi_Iy_N\|_{L_t^{q_{\text{rem}}}L_x^2}
\lesssim T^{\vartheta}a_N,\qquad
 \|\chi_Iy_{\leq N}\|_{L_t^{q_{\text{rem}}}L_x^{r_{\text{rem}}}}
 \lesssim A_I
\end{align}
uniformly in $N$. Next, let
\begin{equation*}
 \ell_N(t)=r_{0,N}(t)-c_a\sum_{L\geq2N}r_{1,L}(t).
\end{equation*}
Using $U_N=\sum_{M\leq N/2}y_M$ and exchanging the triangular sums gives
\begin{align*}
 \sum_NG_2^N
 &=\sum_Nr_{0,N}y_N-c_a\sum_Nr_{1,N}\sum_{M\leq N/2}y_M\\
 &=\sum_My_M\left(r_{0,M}-c_a\sum_{N\geq2M}r_{1,N}\right)=\sum_N\ell_N(t)y_N.
\end{align*}
Now we estimate $\ell_N$. By \eqref{eq:r-bounds}, \eqref{p>2}, and spatial
H\"older,
\begin{equation*}
 |r_{0,N}(t)|+|r_{1,N}(t)|
 \lesssim
 \|y_N(t)\|_{L_x^2}^{\alpha_a}
 \|y_{\leq N}(t)\|_{L_x^{r_{\text{rem}}}}^{a-\alpha_a}.
\end{equation*}
Since $q_{\text{rem}}>a+1>a$, H\"older in time and
\eqref{eq:scalar-yN-qg} imply
\begin{equation*}
 \|\chi_Ir_{j,N}\|_{L_t^{q_{\text{rem}}/a}}
 \lesssim
 T^{\alpha_a\vartheta}a_N^{\alpha_a} A_I^{a-\alpha_a},
 \qquad j=0,1.
\end{equation*}
\eqref{eq:r-bounds}, H\"older and \eqref{eq:scalar-yN-qg} now yield 
\begin{align}
 \|\chi_I\ell_N\|_{L_t^{q_{\text{rem}}/a}}
 \lesssim{}&
 T^{\alpha_a\vartheta}A_I^{a-\alpha_a}
 \left(a_N^{\alpha_a}+\sum_{L\geq2N}a_L^{\alpha_a}\right)
 \label{eq:ellN-qga}
\end{align}
with the bound
\begin{equation}
\begin{aligned}
 a_N^{\alpha_a}+\sum_{L\geq2N}a_L^{\alpha_a}
 \lesssim&
{N^{-\alpha_a(s_c-\delta-\vare)}
 +N^{-\alpha_a s^*}(b_N^{\alpha_a}+\|b_L\|_{\ell_L^2}^{\alpha_a})}\\
\lesssim&
{N^{-\alpha_a(s_c-\delta-\vare)}
 +N^{-\alpha_a s^*}\|b_L\|_{\ell_L^2}^{\alpha_a}}
 \label{eq:scalar-tail-envelope}
\end{aligned}
\end{equation}
deducing from \eqref{eq:scalar-bN-aN}, where we have also used $b_N\lesssim \|b\|_{\ell^\infty}\leq \|b\|_{\ell^2}$. Now apply H\"older, \eqref{eq:scalar-yN-qg}, \eqref{eq:ellN-qga} and \eqref {eq:scalar-tail-envelope} to obtain
\begin{equation*}
\begin{aligned}
 {\|K^+(\chi_I\ell_Ny_N)\|_{Y^{s^*}}\lesssim \|\chi_I\ell_Ny_N\|_{L_t^1 H_x^{s^*}}}
 \lesssim T^{1-\frac{1+a}{q_{\text{rem}}}}\|\chi_I\ell_N\|_{L_t^{q_{\text{rem}}/a}}\|\chi_Iy_N\|_{L_t^{q_{\text{rem}}}H_x^{s^*}}\lesssim s_N,
\end{aligned}
\end{equation*}
where
\begin{align*}
s_N&=T^{1-\frac {a+1}{q_{\text{rem}}}+(1+\alpha_a)\vartheta}A_I^{a-\alpha_a}
 \bigg[
N^{-\alpha_a s_c+(1+\alpha_a)(\delta+\vare)+\vare_0}
 +N^{-\alpha_a s_c+\delta+\vare+(1-\alpha_a)\vare_0}\|b_L\|_{\ell_L^2}^{\alpha_a}\\
 &\qquad\qquad+N^{-\alpha_a(s_c-\delta-\vare)}b_N
 +N^{-\alpha_a s^*}\|b_L\|_{\ell_L^2}^{\alpha_a}b_N
 \bigg].
\end{align*}
By \eqref{eq:delta-epsilon-varepsilon-margins}, the first dyadic exponent is negative and the second is strictly smaller, while the last two terms are summable by Cauchy--Schwarz. Moreover, since $q_{\mathrm{rem}}>a+1$,
\[
 1-\frac{a+1}{q_{\mathrm{rem}}}+(1+\alpha_a)\vartheta>\vartheta.
\]
Summing in $N$ therefore gives \eqref{eq:gauge-rem-est}. This completes the proof.
\end{proof}

\subsection{Conclusion}
\begin{lemma}\label{final4}
There exist $0<\vartheta\ll 1$, $\tilde{\vartheta}\in(0,1)$ and $C,c>0$, such that for any $0<T\ll 1$ there exists some $\Omega_T\subset \Omega$ satisfying $\mathbb P(\Omega_T^c)<\exp(-cT^{-\tilde\vartheta})$, such that for any $\omega\in\Omega_T$ it holds
\begin{equation*}
\begin{aligned}
&\|K^+(\psi\chi_{I_T}G(e^{it\Delta}\phi^\omega+w))\|_{Y^{s^*}}\\
&\quad\lesssim (T^{\vartheta}\|\phi\|_{H_x^{2s_0+s_1}}+\|w\|_{Z^{s^*}})^a(T^\vartheta+\|w\|_{Y^{s^*}})\\
&\qquad +T^{\vartheta}\left(T^{\vartheta}\|\phi\|_{L_x^2}
 +\|w\|_{Z^{s^*}}\right)^{a-\alpha_a}
 \big(1+\|w\|_{Y^{s^*}(I)}\big)^{1+\alpha_a}.
\end{aligned}
\end{equation*}
\end{lemma}

\begin{proof}
For dyadic $K,N$, set
\begin{align*}
\beta_K^N
&:=
\|\psi\chi_{I_T}G_{1,K}^N\|_{(Z^{-s^*})'}^2,
\notag\\
\alpha_M
&:=
T^{2\vartheta}M^{-2\sigma_\delta}+\|w_M\|_{Z^{s^*}}^2,
\notag\\
C_I
&:=
T^{\vartheta}\|\phi\|_{H_x^{2s_0+s_1}}+\|w\|_{Z^{s^*}}.
\end{align*}
Lemma \ref{final1} and \ref{final2} imply that
\begin{equation*}
\begin{aligned}
\beta_K^N
&\lesssim
C_I^{2a}
\left(\frac{K}{N}\right)^{2s^*}
\cdot
\alpha_N,
\\[2mm]
\text{and}\qquad
\beta_K^N
&\lesssim
C_I^{2a}
\left(\frac{N}{K}\right)^{2\mu^*}
\cdot
\sum_{L\le N}
\left(\frac{L}{N}\right)^{2\nu^*}
\alpha_L,
\qquad
4N\le K.
\end{aligned}
\end{equation*}
Since $s^*,\mu^*,\nu^*,\sigma_\delta>0$, this is exactly the condition in \cite[Lem. 4.4]{KwakKwon}. Thus by \cite[(4.9)]{KwakKwon} it holds
\begin{align*}
\|K^+(\psi_1\chi_{I_T}G_1(e^{it\Delta}\phi^\omega+w))\|_{Y^{s^*}}\lesssim 
(T^{\vartheta}\|\phi\|_{H_x^{2s_0+s_1}}+\|w\|_{Z^{s^*}})^a(T^\vartheta+\|w\|_{Y^{s^*}}).
\end{align*}
The desired claim follows by combining the estimate for $G_2$ given by \eqref{eq:gauge-rem-est}. 
\end{proof}

\section{Proof of Theorem \ref{main_thm}}
Throughout this section, we fix a sample $\omega$ in the intersection of
the good events constructed above, so that all preceding pathwise estimates
hold simultaneously.  We abbreviate the random linear solution by
$z=z^\omega=e^{it\Delta}\phi^\omega$.

\subsection{The $V^2$-space and its properties}
We shall prove Theorem \ref{main_thm} based on contraction arguments. The space $V^2$ and its properties will be useful by dealing with the scalar term $\mu(y)$. We first recall the definition of the scalar $V^2$ space. For a bounded interval $J=[t_-,t_+]$ and a scalar function $m:J\to\mathbb C$, define
\begin{equation*}
 [m]_{V^2(J)}
 :=\sup_{t_-=t_0<\cdots<t_K=t_+}
 \left(\sum_{k=1}^K|m(t_k)-m(t_{k-1})|^2\right)^{1/2}
\end{equation*}
and
\[
 \|m\|_{V^2(J)}:=\|m\|_{L^\infty(J)}+[m]_{V^2(J)}.
\]
The following lemma shows that Sobolev spaces and space of absolutely continuous functions are embedded to $V^2$.
\begin{lemma}
\label{lem:AC-V2}
Let $J\subset\mathbb R$ be bounded and let $m:J\to\mathbb C$ be absolutely
continuous.  Then
\begin{equation}\label{eq:AC-to-V2}
 [m]_{V^2(J)}\leq \|m'\|_{L^1(J)}.
\end{equation}
Moreover,
\begin{equation}\label{eq:AC-V2-interpolation}
 [m]_{V^2(J)}^2
 \leq 2\|m\|_{L^\infty(J)}\|m'\|_{L^1(J)}.
\end{equation}
Consequently, if $1<p<\infty$, then
\begin{equation}\label{eq:W1p-holder-V2}
 W^{1,p}(J)
 \hookrightarrow C^{0,1-1/p}(J)\cap V^2(J),
\end{equation}
with
\begin{align}
 [m]_{C^{0,1-1/p}(J)}
 &\leq \|m'\|_{L^p(J)},
 \label{eq:Morrey-scalar}\\
 [m]_{V^2(J)}
 &\leq |J|^{1-1/p}\|m'\|_{L^p(J)}.
 \label{eq:W1p-V2-bound}
\end{align}
\end{lemma}

\begin{proof}
For every partition $\{t_k\}_{k=0}^K$ of $J$, absolute continuity gives
\[
 |m(t_k)-m(t_{k-1})|
 \leq\int_{t_{k-1}}^{t_k}|m'(t)|\,dt.
\]
Using $\ell^2\hookrightarrow\ell^1$ and summing over the disjoint partition
intervals proves \eqref{eq:AC-to-V2}.  Also,
\begin{align*}
 \sum_{k=1}^K|m(t_k)-m(t_{k-1})|^2
 &\leq
 \max_k|m(t_k)-m(t_{k-1})|
 \sum_{k=1}^K|m(t_k)-m(t_{k-1})|\\
 &\leq 2\|m\|_{L^\infty(J)}\|m'\|_{L^1(J)},
\end{align*}
which proves \eqref{eq:AC-V2-interpolation}.  Finally, Holder's inequality
implies
\[
 |m(t)-m(s)|
 \leq |t-s|^{1-1/p}\|m'\|_{L^p(J)}
\]
and
\[
 \|m'\|_{L^1(J)}
 \leq |J|^{1-1/p}\|m'\|_{L^p(J)}.
\]
This proves \eqref{eq:W1p-holder-V2}--\eqref{eq:W1p-V2-bound}.
\end{proof}

Next, set
\begin{equation*}
 p_{\rm rem}:=\frac{q_{\mathrm{rem}}}{a}>1,
 \qquad
 \gamma_{\mathrm{rem}}
 :=1-\frac1{p_{\rm rem}}
 =1-\frac{a}{q_{\mathrm{rem}}}.
\end{equation*}
In particular, by \eqref{eq:qg-choice} it holds $\gamma_{\mathrm{rem}}>
 \max\left\{s_0+\frac{s_1}{2},\frac1p-\sigma\right\}$.

We shall apply Lemma \ref{lem:AC-V2} to the function $e^{-i\Gamma(t)}$. The exact properties we will be using are stated in the following lemma.

\begin{lemma}\label{lem:canonical-phase-regularity}
Let $J\subset I_T$ and suppose
$y\in L_t^{q_{\mathrm{rem}}}L_x^{r_{\mathrm{rem}}}(J\times\mathbb T^d)$.
Define
\begin{equation*}
 \Gamma(t)=\lambda c_a\int_0^t\mu(y)(s)\,ds,
 \qquad
 \rho(t)=e^{-i\Gamma(t)}.
\end{equation*}
Then $\rho\in W^{1,p_{\rm rem}}\cap
C^{\gamma_{\mathrm{rem}}}\cap V^2(J)$ and
\begin{align}
 \|\rho\|_{C^{\gamma_{\mathrm{rem}}}(J)}\lesssim  \|\rho\|_{W^{1,p_{\mathrm{rem}}}(J)}
 &\lesssim
 1+\|y\|_{L_t^{q_{\mathrm{rem}}}L_x^{r_{\mathrm{rem}}}(J)}^a,
 \label{eq:phase-Cgamma-bound}\\
 \|\rho\|_{V^2(J)}
 &\lesssim
 1+|J|^{\gamma_{\mathrm{rem}}}
 \|y\|_{L_t^{q_{\mathrm{rem}}}L_x^{r_{\mathrm{rem}}}(J)}^a.
 \label{eq:phase-V2-bound}
\end{align}
\end{lemma}

\begin{proof}
Notice first that $\rho'(t)=-ie^{-i\Gamma(t)}\Gamma'(t)$. Since $r_{\mathrm{rem}}\geq a$ by
\eqref{p>2}, spatial Holder on the finite-volume torus
gives
\begin{equation*}
 |\Gamma'(t)|
 =|\lambda|c_a|\mu(y)(t)|
 \lesssim \|y(t)\|_{L_x^{r_{\mathrm{rem}}}}^a,
\end{equation*}
implying also that
\begin{equation}\label{eq:phase-derivative-Lp}
  \|\Gamma'\|_{L_t^{p_{\rm rem}}(J)}
 \lesssim
 \|y\|_{L_t^{q_{\mathrm{rem}}}L_x^{r_{\mathrm{rem}}}(J)}^a.
\end{equation}
Since $|\rho|=1$ and $|\rho'|=|\Gamma'|$, \eqref{eq:Morrey-scalar} and \eqref{eq:phase-derivative-Lp} yield \eqref{eq:phase-Cgamma-bound}.  Likewise,
\eqref{eq:W1p-V2-bound} gives
\[
 [\rho]_{V^2(J)}
 \leq |J|^{\gamma_{\mathrm{rem}}}
 \|\rho'\|_{L_t^{p_{\rm rem}}(J)},
\]
which, together with $\|\rho\|_{L^\infty}=1$, proves
\eqref{eq:phase-V2-bound}.
\end{proof}

We next record the multiplier estimates needed in the contraction proof.

\begin{lemma}\label{lem:time-phase-multiplier}
Let $J\subset \R$ be a bounded interval and $m:J\to\mathbb C$.  If
$m\in C^{\gamma_{\mathrm{rem}}}\cap V^2(J)$, then multiplication by $m$ is both bounded on
$Z^s(J)$ and $Y^s(J)$ for every $s\in\mathbb R$.  More precisely, 
\begin{equation}\label{eq:Y-phase-multiplier}
\begin{aligned}
 \|mu\|_{Z^s(J)}
 &\lesssim
 \|m\|_{C^{\gamma_{\mathrm{rem}}}(J)}\|u\|_{Z^s(J)},\\
  \|mu\|_{Y^s(J)}
 &\lesssim
 \|m\|_{V^2(J)}\|u\|_{Y^s(J)}.
\end{aligned}
\end{equation}
\end{lemma}

\begin{proof}
Extend $m$ constantly outside $J$.  The $L_t^qL_x^r$ component of $Z^s$ given in the definition \eqref{def Zs} is
immediate.  Since $m$ depends only on time, we have
\(
 I_{Rk}(mu_N)=mI_{Rk}u_N.
\)
For $0<\alpha<\gamma_{\mathrm{rem}}$, the difference characterization of Besov spaces given by Lemma \ref{lem_triebel} gives
\begin{align*}
 \|mf\|_{B^\alpha_{t,p,1}E}
 &\lesssim
 \|m\|_{L^\infty}\|f\|_{B^\alpha_{t,p,1}E}
 +[m]_{C^{\gamma_{\mathrm{rem}}}}
 \|f\|_{L_t^pE}
 \int_0^1h^{\gamma_{\mathrm{rem}}-\alpha}\,\frac{dh}{h}\\
 &\lesssim
 \|m\|_{C^{\gamma_{\mathrm{rem}}}}
 \|f\|_{B^\alpha_{t,p,1}E},
\end{align*}
where we also used the fact $\gamma_{\rm rem}>1/p-\sigma\gg \alpha$. This completes the $Z^s$ claim in \eqref{eq:Y-phase-multiplier}.

For the $Y^s$ claim, let $\{t_k\}$ be a partition.  Then
\begin{align*}
 m(t_k)u(t_k)-m(t_{k-1})u(t_{k-1})
 &=m(t_k)\big(u(t_k)-u(t_{k-1})\big)\\
 &\quad+\big(m(t_k)-m(t_{k-1})\big)u(t_{k-1}).
\end{align*}
Taking the $\ell^2_k$ norm and using
$\|u\|_{L_t^\infty L_x^2}\lesssim\|u\|_{V^2 L_x^2}$ gives
\begin{align}\label{5.15}
 \|mu\|_{V^2}
 \lesssim \|m\|_{V^2}\|u\|_{V^2 L_x^2}.
\end{align}
The claim follows by applying \eqref{5.15} to each interaction-representation Fourier coefficient in \eqref{def XsYs} and then taking multiplication
by $\langle n\rangle^s$ and summing in $\ell_n^2$.
\end{proof}

The following identity allows us to insert a time-dependent scalar in front
of a nonlinear forcing term while using the already established estimate of
Lemma \ref{final4}.

\begin{lemma}
\label{lem:phase-weighted-duhamel}
Let $J\subset I_T$, $m\in W^{1,1}(J)\cap V^2(J)$ and
$$V:=K^+(\psi\chi_J f).$$  
Then for $t\in J$ it holds
\begin{equation}\label{eq:phase-Duhamel-identity}
 K^+(\psi\chi_Jmf)
 =mV-K^+(\psi\chi_Jm'V).
\end{equation}
Consequently, for every $s\in\mathbb R$, we have
\begin{equation}\label{eq:phase-Duhamel-bound}
 \|K^+(\psi\chi_Jmf)\|_{Y^s(J)}
 \lesssim
 \big(\|m\|_{V^2(J)}+\|m'\|_{L^1(J)}\big)
 \|K^+(\psi\chi_Jf)\|_{Y^s(J)}.
\end{equation}
\end{lemma}

\begin{proof}
Extend $m$ constantly outside $J$.  By the definition of the Schr\"odinger operator $K^+$, it holds
\[
 (i\partial_t+\Delta)V=i\psi\chi_Jf,
 \qquad V(0)=0.
\]
Hence using also product rule we obtain
\[
 (i\partial_t+\Delta)(mV)
 =i\psi\chi_J(mf+m'V),
 \qquad (mV)(0)=0,
\]
which proves \eqref{eq:phase-Duhamel-identity}.  Lemma
\ref{lem:time-phase-multiplier}, the estimate
$K^+:L_t^1H_x^s\to Y^s$, and $Y^s\hookrightarrow L_t^\infty H_x^s$ give
\begin{align*}
 \|mV\|_{Y^s(J)}
 &\lesssim \|m\|_{V^2(J)}\|V\|_{Y^s(J)},\\
 \|K^+(\psi\chi_Jm'V)\|_{Y^s(J)}
 &\lesssim \|m'V\|_{L_t^1H_x^s(J)}
 \lesssim \|m'\|_{L^1(J)}\|V\|_{Y^s(J)},
\end{align*}
thus proving \eqref{eq:phase-Duhamel-bound}.
\end{proof}

\subsection{Heuristics for the contraction decomposition}\label{sec 5.3}
This subsection is devoted to give an explanation for the different terms appearing in the contraction mapping which we will be considering later. Let $u$ be a solution of the original NLS \eqref{eq:nls} and let $y=z+w$ be the solution of the corresponding gauged NLS. Then
$$ u=\beta y=\beta(z+w)=:\beta z+v$$
with $v=\beta w$ and $\beta=\beta_u=\exp(-i\lambda c_a\int_0^t\mu(y)(s)\,ds)$. Hence, if $\Phi$ denotes the contraction mapping for the original NLS, then heuristically it holds
\[
\Phi(u)
=
z
-i\lambda K^+\bigl(\psi\chi_{I_T}F(u)\bigr).
\]
We now want to express $F(u)=F(\beta z+v)$ using the gauged transformed nonlinearity $G$, so that we are able to utilize the nonlinear estimates established in Section \ref{nonlin_sec}. Indeed, using $F=G+c_a\mu(\,\cdot\,)\,\cdot$ and $G(\beta f)=m_\beta G(f)$ it holds
\[
F(u)
=
m_\beta G\bigl(z+\beta^{-1}v\bigr)
+c_a\mu(u)\beta z
+c_a\mu(u)v,
\]
where $m_\beta:=|\beta|^a\beta$. Since
\(
e^{i(t-s)\Delta}z(s)=z(t)
\)
and $\psi\equiv1$ on $I_T$, the second summand's Duhamel contribution satisfies
\[
K^+\bigl(\psi\chi_{I_T}\mu(u)\beta z\bigr)(t)
=
z(t)\int_0^t\mu(u)(s)\beta(s)\,ds.
\]
It is therefore natural to combine this contribution with the random linear solution $z$ and define the new coefficient
\begin{equation*}
 \beta_+(t)
 :=1-i\lambda c_a\int_0^t\mu(u)(s)\beta(s)\,ds
\end{equation*}
The remaining terms are collected in the smoother component
\begin{equation*}
 v_+
 :=-i\lambda K^+\left(\psi\chi_{I_T}
 \left[m_\beta G(z+\beta^{-1}v)+c_a\mu(u)v\right]\right).
\end{equation*}
By construction, it then holds
\begin{align}\label{intertwining}
\Phi(u)=\beta_+z^\omega+v_+.
\end{align}
Before reaching a fixed point, $\beta$ is only a trial coefficient.
At a fixed point, uniqueness of the decomposition (see Lemma \ref{lem:phase-metric-complete} below) into a multiple of
$z^\omega$ and a $Y^{s^*}$ remainder gives
\[
\beta_+=\beta,
\qquad
v_+=v.
\]
Consequently,
\[
\beta'(t)
=
-i\lambda c_a\mu(u)(t)\beta(t),
\qquad
\beta(0)=1,
\]
and hence
\[
\beta(t)
=
\exp\left(
-i\lambda c_a\int_0^t\mu(u)(s)\,ds
\right).
\]
Therefore the trial coefficient becomes precisely the canonical
inverse-gauge phase at the fixed point.

\subsection{Construction of the metric space}
We now give the precise construction of the underlying metric space where we shall apply the Banach fixed point theorem. Let ${B}\geq1$ be fixed.  Define
\begin{equation*}
 \mathcal P_T:=
 \left\{\beta\in W^{1,p_{\mathrm{rem}}}(I_T):
 \begin{array}{l}
 \beta(0)=1,\\[1mm]
 \|\beta-1\|_{L_t^\infty(I_T)}\leq\frac12,\\[1mm]
 \|\beta'\|_{L_t^{p_{\mathrm{rem}}}(I_T)}\leq {B}
 \end{array}
 \right\}.
\end{equation*}
By Lemma \ref{lem:AC-V2}, every $\beta\in\mathcal P_T$ satisfies
\begin{equation}\label{eq:trial-phase-uniform-multipliers}
 \|\beta\|_{C^{\gamma_{\mathrm{rem}}}(I_T)}
 +\|\beta\|_{V^2(I_T)}
 +\|\beta^{-1}\|_{C^{\gamma_{\mathrm{rem}}}(I_T)}
 +\|\beta^{-1}\|_{V^2(I_T)}
 \lesssim_{B}1.
\end{equation}
Indeed, $1/2\leq|\beta|\leq3/2$, and composition with the smooth map
$z\mapsto z^{-1}$ on this annulus preserves all the displayed bounds.  The
same observation applies to $ m_\beta:=|\beta|^a\beta$, thus
\begin{equation}\label{eq:m-beta-uniform}
 \|m_\beta\|_{C^{\gamma_{\mathrm{rem}}}}
 +\|m_\beta\|_{V^2}
 +\|m_\beta'\|_{L^{p_{\mathrm{rem}}}}
 \lesssim_{B}1.
\end{equation}

For $R>0$, set
\begin{equation*}
 \mathfrak X_{T,R}:=
 \left\{(\beta,v):
 \beta\in\mathcal P_T,\quad
 v\in Y^{s^*}(I_T),\quad
 v(0)=0,\quad
 \|v\|_{Y^{s^*}(I_T)}\leq R
 \right\}
\end{equation*}
and define the reconstruction map
\begin{equation*}
 \mathcal J(\beta,v):=\beta z+v
\end{equation*}
and the weak pullback distance
\begin{equation}\label{eq:phase-adapted-metric}
 d\big((\beta_1,v_1),(\beta_2,v_2)\big)
 :=\|\mathcal J(\beta_1,v_1)-\mathcal J(\beta_2,v_2)\|_{Y^0(I_T)}.
\end{equation}
We verify that \eqref{eq:phase-adapted-metric} is a complete metric.

\begin{lemma}
\label{lem:phase-metric-complete}
The space $\mathfrak X_{T,R}$ with the underlying metric $d$ is complete.
\end{lemma}

\begin{proof}
To show that $(\mathfrak X_{T,R}, d)$ is metric space, it suffices to show that $d((\beta_1,v_1),(\beta_2,v_2))=0$ implies $(\beta_1,v_1)=(\beta_2,v_2)$. Indeed, the condition $d((\beta_1,v_1),(\beta_2,v_2))=0$ and the definition of the mapping $\mathcal{J}$ imply that
\begin{equation*}
 (\beta_1(t)-\beta_2(t))z(t)=v_2(t)-v_1(t)
\end{equation*}
for every $t\in I_T$.  The right-hand side belongs to $H^{s^*}\hookrightarrow H^{s_c}$, whereas
$z(t)\notin H^{s_c}$ by Lemma \ref{lem:endpoint-roughness}.  Thus $\beta_1(t)=\beta_2(t)$ for every $t$, and then
$v_1=v_2$.

It remains to show the completeness of $(\mathfrak X_{T,R}, d)$. Let $X_n=(\beta_n,v_n)$ be a $d$-Cauchy sequence and set
$u_n=\mathcal J(X_n)$.  Since $Y^0(I_T)$ is Banach, there exists
$u\in Y^0(I_T)$ such that
\begin{equation}\label{eq:un-Y0-limit}
 u_n\longrightarrow u
 \qquad\text{in }Y^0(I_T).
\end{equation}
The sequence $\{\beta_n\}$ is bounded in $W^{1,p_{\mathrm{rem}}}(I_T)$.  By 
Morrey's compact embedding 
$$W^{1,p_{\mathrm{rem}}}(I_T)\hookrightarrow\hookrightarrow L^\infty(I_T)$$ 
and weakly lower semi-continuity of a norm, after passing to a subsequence, we have
\begin{equation*}
 \beta_n\longrightarrow\beta
 \quad\text{uniformly on }I_T,
 \qquad
 \beta_n\rightharpoonup\beta
 \quad\text{in }W^{1,p_{\mathrm{rem}}}(I_T)
\end{equation*}
and $\|\beta'\|_{L_t^{p_{\mathrm{rem}}}(I_T)}\leq {B}$ for some $\beta\in W^{1,p_{\mathrm{rem}}}(I_T)$, hence $\beta\in\mathcal P_T$.  Applying
\eqref{eq:AC-V2-interpolation} to $\beta_n-\beta$ gives
\begin{equation*}
 [\beta_n-\beta]_{V^2(I_T)}^2
 \leq
 2\|\beta_n-\beta\|_{L^\infty(I_T)}
 \|\beta_n'-\beta'\|_{L^1(I_T)}
 \longrightarrow0,
\end{equation*}
because the second factor is uniformly bounded by $C{B}$.  Since
$z\in Y^0(I_T)$, Lemma \ref{lem:time-phase-multiplier} yields
\begin{equation*}
 (\beta_n-\beta)z\longrightarrow0
 \qquad\text{in }Y^0(I_T).
\end{equation*}
It follows from \eqref{eq:un-Y0-limit} that
\begin{equation*}
 v_n=u_n-\beta_nz
 \longrightarrow v:=u-\beta z
 \qquad\text{in }Y^0(I_T).
\end{equation*}
For each spatial frequency $k$, this implies convergence of the corresponding
interaction-representation coefficient in $V^2$.  Fatou's lemma therefore
gives
\begin{align*}
 \|v\|_{Y^{s^*}(I_T)}^2
 &\leq\liminf_{n\to\infty}
 \|v_n\|_{Y^{s^*}(I_T)}^2
 \leq R^2.
\end{align*}
Moreover, the embedding $Y^0(I_T)\hookrightarrow L_t^\infty L_x^2(I_T)$ implies
$v_n(0)\to v(0)$ in $L_x^2$.  Since $v_n(0)=0$, we have $v(0)=0$.  Hence
$(\beta,v)\in\mathfrak X_{T,R}$.  Finally,
\[
 d(X_n,(\beta,v))
 =\|u_n-u\|_{Y^0(I_T)}\longrightarrow0,
\]
implying that every $d$-Cauchy sequence converges in $\mathfrak X_{T,R}$. This completes the proof.
\end{proof}

\subsection{Construction of the contraction mapping}
For $X=(\beta,v)\in\mathfrak X_{T,R}$, put
\begin{equation}\label{eq:state-u-y-w}
 u=\mathcal J(X)=\beta z+v,
 \qquad
 w=\beta^{-1}v,
 \qquad
 y=z+w,
\end{equation}
so that $u=\beta y$.  By \eqref{eq:trial-phase-uniform-multipliers},
\begin{equation*}
 \|w\|_{Y^{s^*}(I_T)}+\|w\|_{Z^{s^*}(I_T)}
 \lesssim R.
\end{equation*}
As explained in the preceding subsection, define 
\begin{equation}\label{eq:beta-plus-def}
 \beta_+(t)
 :=1-i\lambda c_a\int_0^t\mu(u)(s)\beta(s)\,ds
\end{equation}
and
\begin{equation*}
 v_+
 :=-i\lambda K^+\left(\psi\chi_{I_T}
 \left[m_\beta G(z+\beta^{-1}v)+c_a\mu(u)v\right]\right).
\end{equation*}
We then set our contraction mapping by
\begin{equation*}
 \mathcal T(\beta,v):=(\beta_+,v_+).
\end{equation*}

In the following we first show that $\mathcal T$ is a self-map on $\mathfrak X_{T,R}$. Fix a sufficiently large structural constant $C_0=C_0(B)\geq1$ and define
\begin{equation*}
 \delta_T
 :=C_0T^\vartheta
 (1+\|\phi\|_{H_x^{2s_0+s_1}}).
\end{equation*}
On a good event, the probabilistic estimates given previously imply
\begin{equation}\label{eq:random-small-coeff-global}
 \mathcal R_J:=
 \|\psi\chi_Jz\|_{L_t^rB_{x,r,2}^{2s_0+s_1}}
 +\|\psi\chi_Jz\|_{B_{t,r,2}^{s_0+s_1/2}L_x^r}\lesssim\delta_T.
\end{equation}

\begin{lemma}
\label{lem:phase-self-map}
There exists $0<T\ll 1$ such that with $R=2\delta_T$ we have $\mathcal T(\mathfrak X_{T,R})\subset\mathfrak X_{T,R}$.
\end{lemma}

\begin{proof}
Fix $X=(\beta,v)\in\mathfrak X_{T,R}$ and use the notation
\eqref{eq:state-u-y-w}.  The estimates used in
\eqref{eq:scalar-yN-qg}, together with
Lemma \ref{lem:time-phase-multiplier}, give
\begin{equation}\label{eq:u-qrem-bound}
 \|u\|_{L_t^{q_{\mathrm{rem}}}L_x^{r_{\mathrm{rem}}}(I_T)}
 \lesssim
(T^{\vartheta}\|\phi\|_{L_x^2}+R)\lesssim \delta_T.
\end{equation}
From \eqref{eq:beta-plus-def} we have
\(
 \beta_+'=-i\lambda c_a\mu(u)\beta.
\)
Since $r_{\mathrm{rem}}\geq a$, spatial H\"older first gives
\[|\mu(u)(t)|\sim \|u(t)\|_{L_x^a}^a\lesssim \|u(t)\|_{L_x^{r_{\rm rem}}}^a. \]
Combining $\|\beta\|_{L_t^\infty}\leq 3/2$ it follows $|\beta_+'(t)|\lesssim \|u(t)\|_{L_x^{r_{\rm rem}}}^a$. Consequently,
\eqref{eq:u-qrem-bound} yields
\begin{align*} 
 \|\beta_+'\|_{L_t^{p_{\mathrm{rem}}}(I_T)}\lesssim \|u\|_{L_t^{q_{\mathrm{rem}}}L_x^{r_{\rm rem}}(I_T)}^a\lesssim \delta_T^a.
\end{align*}
By the fundamental theorem of calculus, $\beta_+(t)-1=
\int_0^t
\beta_+'(s)\,ds$. Hence
\begin{align*}
\begin{aligned}
\|\beta_+-1\|_{L_t^\infty(I_T)}
\lesssim
\int_{J_t}
|\beta_+'(s)|\,ds
\lesssim
T^{\gamma_{\mathrm{rem}}}
\|\beta_+'\|_{L_t^{p_{\mathrm{rem}}}(I_T)}\lesssim T^{\gamma_{\mathrm{rem}}}\delta_T^a.
\end{aligned}
\end{align*}
After decreasing $T$, these estimates imply
$\beta_+\in\mathcal P_T$.

We now estimate $v_+$.  By \eqref{eq:m-beta-uniform}, Lemma
\ref{lem:phase-weighted-duhamel}, and Lemma \ref{final4},
\begin{align}
\begin{aligned}
 &{\left\|K^+\left(\psi\chi_{I_T}
 m_\beta G(z+\beta^{-1}v)\right)\right\|_{Y^{s^*}(I_T)}}
 \\
 &\quad\lesssim_B
 \left(\delta_T+CR\right)^a\left(\delta_T+CR\right)
 \\
 &\qquad
 +T^\vartheta\left(\delta_T+CR\right)^{a-\alpha_a}
 \left(1+CR\right)^{1+\alpha_a}.
\end{aligned}
\label{eq:weighted-G-self-map}
\end{align}
Let $C_B$ denote the implicit constant in \eqref{eq:weighted-G-self-map}. The first term on its right-hand side is $O_B(\delta_T^{a+1})$. If $0<a\leq1$, then $\alpha_a=a$, and
\[
 C_BT^\vartheta(1+CR)^{1+a}
 =\frac{C_B}{C_0}
 \frac{(1+CR)^{1+a}}{1+\|\phi\|_{H_x^{2s_0+s_1}}}
 \delta_T.
\]
Choose $C_0$ sufficiently large and then $T$ sufficiently small so that this term is at most $3\delta_T/2$; the first term is then at most $\delta_T/4$. If $a>1$, then $\alpha_a=1$, and the second term is $O_B(\delta_T^a)$.  For the remaining scalar term, Lemma \ref{basis duhamel} gives
\begin{equation}
\begin{aligned}
 {\|K^+(\psi\chi_{I_T}\mu(u)v)\|_{Y^{s^*}(I_T)}}
 &\lesssim
 {\|\mu(u)v\|_{L_t^1H_x^{s^*}(I_T)}}
\lesssim
 \|\mu(u)\|_{L_t^1(I_T)}
 {\|v\|_{L_t^\infty H_x^{s^*}(I_T)}}
 \\
 &\lesssim
 T^{\gamma_{\mathrm{rem}}}\delta_T^aR\lesssim \delta_T^{a+1}.
\end{aligned}
\label{eq:mu-v-self-map}
\end{equation}
The right-hand side of \eqref{eq:mu-v-self-map} is $O(\delta_T^{a+1})$. Hence choosing $T\ll 1$ yields ${\|v_+\|_{Y^{s^*}}}\leq 2\delta_T=R$. Since also $v_+(0)=0$, the proof is complete.
\end{proof}

The next lemma gives the smallness needed for the mapping $\mathcal T$ to be a contraction.

\begin{lemma}
\label{lem:phased-difference-estimate}
Let $J\subset I_T$ and
\[
 U_j=\rho_j(t)z+v_j,
 \qquad
 \rho_j\in C^{\gamma_{\mathrm{rem}}}(J),
 \qquad
 v_j\in Z^{s_c}(J),
 \qquad j=1,2.
\]
Set
\begin{equation*}
 \mathcal A_J
 :=\mathcal R_J(\omega)
 \max_{j=1,2}\|\rho_j\|_{C^{\gamma_{\mathrm{rem}}}(J)}
 +\sum_{j=1}^2\|v_j\|_{Z^{s_c}(J)}.
\end{equation*}
If $U_1-U_2\in Z^0(J)$, then
\begin{equation}\label{eq:phased-difference-estimate}
 \|\psi\chi_J(F(U_1)-F(U_2))\|_{(Z^0)'}
 \lesssim
 \mathcal A_J^a\|U_1-U_2\|_{Z^0(J)}.
\end{equation}
\end{lemma}

\begin{proof}
Let $h=U_1-U_2$ and $U_\tau=U_2+\tau h$.  Then
\[
 U_\tau=\rho_\tau z+v_\tau,
 \qquad
 \rho_\tau=(1-\tau)\rho_2+\tau\rho_1,
 \qquad
 v_\tau=(1-\tau)v_2+\tau v_1.
\]
Lemma \ref{lem:time-phase-multiplier}, \eqref{eq:random-small-coeff-global},
and the deterministic part of the
$Z^{s_c}$ embedding give, uniformly in $\tau\in[0,1]$,
\begin{align}\label{eq:phased-base-norm}
 &\|\psi\chi_JU_\tau\|_{L_t^rB_{x,r,2}^{2s_0+s_1}}
 +\|\psi\chi_JU_\tau\|_{B_{t,r,2}^{s_0+s_1/2}L_x^r}
 \lesssim\mathcal A_J.
\end{align}
Lemma
\ref{lem:mixed-coeff}, with \eqref{eq:phased-base-norm} as input, yields
\begin{equation}\label{eq:phased-half-coeff}
\begin{aligned}
 &\|\psi\chi_J|U_\tau|^{a/2}\|_{
 B^\zeta_{t,\widehat r,\widehat r}
 B^\eta_{x,\widehat r,\widehat r}}\\
 &\quad+
 \|\psi\chi_J|U_\tau|^{a/2-1}U_\tau\|_{
 B^\zeta_{t,\widehat r,\widehat r}
 B^\eta_{x,\widehat r,\widehat r}}
 \lesssim\mathcal A_J^{a/2}.
\end{aligned}
\end{equation}
We also use the deterministic consequence of
\cite[Lem. 3.10]{KwakKwon}
\begin{equation}\label{eq:deterministic-Z0-bilinear}
 \|\psi Au^*\|_{L_{t,x}^2}
 \lesssim
 \|u\|_{Z^0}
 \|A\|_{B^\zeta_{t,\widehat r,\widehat r}
 B^\eta_{x,\widehat r,\widehat r}},
 \qquad u^*\in\{u,\overline u\}.
\end{equation}
Let $g\in Z^0(J)$ with $\|g\|_{Z^0(J)}=1$.  The Wirtinger formula gives
\[
 F(U_1)-F(U_2)
 =h\int_0^1c_a|U_\tau|^a\,d\tau
 +\overline h\int_0^1\frac a2
 |U_\tau|^{a-2}U_\tau^2\,d\tau.
\]
Using \eqref{eq:deterministic-Z0-bilinear} and
\eqref{eq:phased-half-coeff},
\begin{align*}
 \left|\int\psi^2\chi_Jh|U_\tau|^a\overline g\,dxdt\right|
 &\leq
 \|\psi\chi_Jh|U_\tau|^{a/2}\|_{L^2}
 \|\psi\chi_Jg|U_\tau|^{a/2}\|_{L^2}\\
 &\lesssim\mathcal A_J^a\|h\|_{Z^0(J)},\\
 \left|\int\psi^2\chi_J\overline h
 |U_\tau|^{a-2}U_\tau^2\overline g\,dxdt\right|
 &\leq
 \|\psi\chi_J\overline h
 |U_\tau|^{a/2-1}U_\tau\|_{L^2}\\
 &\quad\times
 \|\psi\chi_J\overline g
 |U_\tau|^{a/2-1}U_\tau\|_{L^2}\\
 &\lesssim\mathcal A_J^a\|h\|_{Z^0(J)}.
\end{align*}
Integrating in $\tau$ and taking the supremum over $g$ proves
\eqref{eq:phased-difference-estimate}.
\end{proof}

\subsection{Conclusion}
We are now ready to give the desired proof of Theorem \ref{main_thm}

\begin{proof}[Proof of Theorem \ref{main_thm}]
Write $X_j=(\beta_j,v_j)$ and $u_j=\mathcal JX_j$.  By
\eqref{eq:trial-phase-uniform-multipliers},
$\|\beta_j\|_{C^{\gamma_{\mathrm{rem}}}}\lesssim_B1$, while
$Y^{s^*}\hookrightarrow Z^{s^*}\hookrightarrow Z^{s_c}$ gives $\|v_j\|_{Z^{s_c}}\lesssim R$.
Hence \eqref{eq:random-small-coeff-global} implies that the coefficient in
Lemma \ref{lem:phased-difference-estimate} satisfies
\begin{equation*}
 \mathcal A_{I_T}
 \lesssim_B\delta_T.
\end{equation*}
Using \eqref{intertwining}, Lemma
\ref{lem:3.6}, Lemma \ref{lem:phased-difference-estimate} and the embedding
$Y^0\hookrightarrow Z^0$, we obtain
\begin{align*}
 d(\mathcal TX_1,\mathcal TX_2)
 &=\|\Phi(u_1)-\Phi(u_2)\|_{Y^0(I_T)}\\
 &\lesssim
 \|\psi\chi_{I_T}(F(u_1)-F(u_2))\|_{(Z^0)'}\\
 &\lesssim
 \mathcal A_{I_T}^a\|u_1-u_2\|_{Z^0(I_T)}\\
 &\lesssim
 C_B\delta_T^a
 d(X_1,X_2).
\end{align*}
Hence decreasing $T$ if necessary, the mapping
$\mathcal T$ is a strict contraction on
$(\mathfrak X_{T,R},d)$. By Lemma \ref{lem:phase-metric-complete} and Lemma \ref{lem:phase-self-map}, Banach's fixed point theorem gives a unique
\begin{equation*}
 (\beta,v)\in\mathfrak X_{T,R}
 \qquad\text{such that}\qquad
 \mathcal T(\beta,v)=(\beta,v).
\end{equation*}
Let $u=\mathcal J(\beta,v)=\beta z+v$. The intertwining identity shows that
\begin{equation}\label{eq:u-ungauged-mild}
 u=z-i\lambda K^+(\psi\chi_{I_T}F(u))
 \qquad\text{on }I_T,
\end{equation}
hence $u$ solves the ungauged equation \eqref{eq:nls} on $I_T$.
Furthermore, the fixed point identity for $\beta$ gives
\begin{equation}\label{eq:fixed-phase-ODE}
 \beta'=-i\lambda c_a\mu(u)\beta,
 \qquad
 \beta(0)=1.
\end{equation}
Since $\mu(u)$ is real-valued,
\[
 \frac{d}{dt}|\beta|^2
 =2\operatorname{Re}(\overline\beta\beta')=0.
\]
This yields
\begin{equation*}
 |\beta(t)|=1,
 \qquad
 \beta(t)=\exp\left(-i\lambda c_a
 \int_0^t\mu(u)(s)\,ds\right).
\end{equation*}
Now define
\begin{equation*}
 y=\beta^{-1}u=z+w,
 \qquad
 w=\beta^{-1}v.
\end{equation*}
By Lemma \ref{lem:time-phase-multiplier}, $w\in Y^{s^*}(I_T)\hookrightarrow Z^{s^*}(I_T)$ and $\|w\|_{Y^{s^*}(I_T)}\lesssim R$. Since $|u|=|y|$, we have $\mu(u)=\mu(y)$.  Combining
\eqref{eq:u-ungauged-mild} and \eqref{eq:fixed-phase-ODE}, or equivalently
reversing the calculation in \eqref{eq:inverse-gauge}, shows that
\[
 (i\partial_t+\Delta)y
 =\lambda\big(F(y)-c_a\mu(y)y\big)
 =\lambda G(y),
 \qquad y(0)=\phi^\omega.
\]
This proves the existence of a solution $y$ of the gauged NLS in the class asserted in Theorem \ref{main_thm}.

Banach's theorem gives uniqueness inside the small phase-adapted ball.  We now
prove uniqueness in the full class
$z+Y^{s^*}(I_T)$ claimed in Theorem \ref{main_thm} by a local zero-order
absorption argument. First note that, for every fixed $t_0\in I_T$, we have 
\begin{align}\label{eq:local-random-coeff-vanishing}
\mathcal R_J\longrightarrow0 \qquad\text{whenever } J\downarrow\{t_0\}.
\end{align} 
Indeed, the first component of $\mathcal R_J$ has this property by absolute continuity
of the $L_t^r$ norm, and the second follows from Lemma \ref{lem:time-cutoff}.  Now let
\[
 y_j=z+w_j,
 \qquad w_j\in Y^{s^*}(I_T),
 \qquad j=1,2,
\]
be two solutions of \eqref{eq:gauged} with the same initial value.  Define
\begin{equation*}
 \Gamma_j(t)=\lambda c_a\int_0^t\mu(y_j)(s)\,ds,
 \qquad
 \rho_j=e^{-i\Gamma_j},
 \qquad
 u_j=\rho_jy_j.
\end{equation*}
Then each $u_j$ solves
\begin{equation*}
 (i\partial_t+\Delta)u_j=\lambda F(u_j),
 \qquad u_j(0)=\phi^\omega.
\end{equation*}
The estimates used in \eqref{eq:scalar-yN-qg}, together with
$Y^{s^*}\hookrightarrow Z^{s^*}\hookrightarrow Z^{s_c}$, imply
\(
 y_j\in L_t^{q_{\mathrm{rem}}}L_x^{r_{\mathrm{rem}}}(I_T).
\)
Lemma \ref{lem:canonical-phase-regularity} therefore gives, for every
$J\subset I_T$,
\begin{equation}\label{eq:individual-phase-bounds}
 \|\rho_j\|_{C^{\gamma_{\mathrm{rem}}}(J)}
 \lesssim
 1+\|y_j\|_{L_t^{q_{\mathrm{rem}}}L_x^{r_{\mathrm{rem}}}(J)}^a,
 \qquad
 \|\rho_j\|_{V^2(J)}
 \lesssim
 1+|J|^{\gamma_{\mathrm{rem}}}
 \|y_j\|_{L_t^{q_{\mathrm{rem}}}L_x^{r_{\mathrm{rem}}}(J)}^a.
\end{equation}
In particular, Lemma \ref{lem:time-phase-multiplier} gives
\begin{equation*}
 u_j=\rho_jz+v_j,
 \qquad
 v_j:=\rho_jw_j\in
 Y^{s^*}(I_T)\cap Z^{s^*}(I_T).
\end{equation*}
Also $u_j\in Y^0(I_T)$ because $z\in Y^0(I_T)$ and $s_c>0$.
The mild formulation then implies
$u_j\in C(I_T;L_x^2)$.

Let $t_0\in I_T$ be such that $u_1(t_0)=u_2(t_0)$, and let $J\subset I_T$
be a short interval containing $t_0$. Set $h=u_1-u_2$ and
\[
 K_{t_0}^+f(t):=\int_{t_0}^t e^{i(t-s)\Delta}f(s)\,ds.
\]
Since $h(t_0)=0$, the difference equation gives
\[
 h=-i\lambda K_{t_0}^+\bigl(\chi_J(F(u_1)-F(u_2))\bigr)
 \qquad\text{on }J.
\]
The time-translated form of Lemma \ref{basis duhamel}, followed by Lemma
\ref{lem:phased-difference-estimate}, yields
\begin{align}
 \|h\|_{Y^0(J)}
 &\lesssim
 \|\psi\chi_J(F(u_1)-F(u_2))\|_{(Z^0)'}
 \notag\\
 &\lesssim
 \mathcal A_J^a\|h\|_{Z^0(J)}
 \lesssim
 \mathcal A_J^a\|h\|_{Y^0(J)},
 \label{eq:local-ungauged-uniqueness}
\end{align}
where
\begin{equation*}
 \mathcal A_J
 =\mathcal R_J(\omega)
 \max_{j=1,2}\|\rho_j\|_{C^{\gamma_{\mathrm{rem}}}(J)}
 +\sum_{j=1}^2\|\rho_jw_j\|_{Z^{s_c}(J)}.
\end{equation*}
By \eqref{eq:local-random-coeff-vanishing}, the time-translated form of
\eqref{3.13}, Lemma \ref{lem:time-phase-multiplier}, and
\eqref{eq:individual-phase-bounds}, we know that
 $\mathcal A_J\longrightarrow0$
as $J\downarrow\{t_0\}$.
Choose $J$ so that the implicit constant in
\eqref{eq:local-ungauged-uniqueness} times $\mathcal A_J^a$ is less than
$1/2$.  Then $h=0$ on $J$. Starting at $t_0=0$ and propagating this local equality successively to the
right and to the left covers the compact interval $I_T$; hence
$u_1=u_2$ on $I_T$.  Finally, $|y_j|=|u_j|$, and therefore
\[
 \Gamma_j'(t)=\lambda c_a\mu(u_j)(t),
 \qquad \Gamma_j(0)=0.
\]
Thus $u_1=u_2$ implies $\Gamma_1=\Gamma_2$, and then $y_1=y_2$.
This proves uniqueness in $z+Y^{s^*}(I_T)$.  Finally, applying the
preceding high-probability construction to $T_n=2^{-n}$ and using
Borel--Cantelli yields a probability-one event on which an admissible
positive lifespan exists.  This completes the proof of Theorem
\ref{main_thm}.
\end{proof}


\subsubsection*{Acknowledgements}
The author was supported by the NSF grant of Guangdong (No. 2024A1515010497), the QB-Program of Guangdong (No. 2024QN11X141) and the NSF grant of China (No. 12301301).

\subsubsection*{Data availability}
Data sharing is not applicable to this article as no datasets were generated or analysed during the current study.

\subsubsection*{Conflict of interest}

The author declares that he has no conflict of interest.


\begin{thebibliography}{10}

\bibitem{amannbesov}
{\sc Amann, H.}
\newblock Operator-valued {Fourier} multipliers, vector-valued {Besov} spaces,
  and applications.
\newblock {\em Math. Nachr. 186\/} (1997), 5--56.

\bibitem{AmannEmbedding}
{\sc Amann, H.}
\newblock Compact embeddings of vector-valued {S}obolev and {B}esov spaces.
\newblock {\em Glas. Mat. Ser. III 35(55)}, 1 (2000), 161--177.
\newblock Dedicated to the memory of Branko Najman.

\bibitem{Bourgain1}
{\sc Bourgain, J.}
\newblock Fourier transform restriction phenomena for certain lattice subsets
  and applications to nonlinear evolution equations. {I}. {S}chr\"{o}dinger
  equations.
\newblock {\em Geom. Funct. Anal. 3}, 2 (1993), 107--156.

\bibitem{Bourgain2}
{\sc Bourgain, J.}
\newblock Fourier transform restriction phenomena for certain lattice subsets
  and applications to nonlinear evolution equations. {II}. {T}he
  {K}d{V}-equation.
\newblock {\em Geom. Funct. Anal. 3}, 3 (1993), 209--262.

\bibitem{BourgainProb1}
{\sc Bourgain, J.}
\newblock Periodic nonlinear {Schr{\"o}dinger} equation and invariant measures.
\newblock {\em Commun. Math. Phys. 166}, 1 (1994), 1--26.

\bibitem{BourgainProb2}
{\sc Bourgain, J.}
\newblock Invariant measures for the 2d-defocusing nonlinear {Schr{\"o}dinger}
  equation.
\newblock {\em Commun. Math. Phys. 176}, 2 (1996), 421--445.

\bibitem{BDNY24}
{\sc Bringmann, B., Deng, Y., Nahmod, A.~R., and Yue, H.}
\newblock Invariant {Gibbs} measures for the three dimensional cubic nonlinear
  wave equation.
\newblock {\em Invent. Math. 236}, 3 (2024), 1133--1411.

\bibitem{Burq1}
{\sc Burq, N., G\'{e}rard, P., and Tzvetkov, N.}
\newblock Strichartz inequalities and the nonlinear {S}chr\"{o}dinger equation
  on compact manifolds.
\newblock {\em Amer. J. Math. 126}, 3 (2004), 569--605.

\bibitem{Burq2}
{\sc Burq, N., G\'{e}rard, P., and Tzvetkov, N.}
\newblock Bilinear eigenfunction estimates and the nonlinear {S}chr\"{o}dinger
  equation on surfaces.
\newblock {\em Invent. Math. 159}, 1 (2005), 187--223.

\bibitem{Burq3}
{\sc Burq, N., G\'{e}rard, P., and Tzvetkov, N.}
\newblock Multilinear eigenfunction estimates and global existence for the
  three dimensional nonlinear {S}chr\"{o}dinger equations.
\newblock {\em Ann. Sci. \'{E}cole Norm. Sup. (4) 38}, 2 (2005), 255--301.

\bibitem{BurqTzvetkov1}
{\sc Burq, N., and Tzvetkov, N.}
\newblock Random data {Cauchy} theory for supercritical wave equations {I}:
  {Local} theory.
\newblock {\em Invent. Math. 173}, 3 (2008), 449--475.

\bibitem{BurqTzvetkov2}
{\sc Burq, N., and Tzvetkov, N.}
\newblock Random data {Cauchy} theory for supercritical wave equations. {II}.
  {A} global existence result.
\newblock {\em Invent. Math. 173}, 3 (2008), 477--496.

\bibitem{Cazenave2003}
{\sc Cazenave, T.}
\newblock {\em Semilinear {S}chr\"odinger equations}, vol.~10 of {\em Courant
  Lecture Notes in Mathematics}.
\newblock New York University, Courant Institute of Mathematical Sciences, New
  York; American Mathematical Society, Providence, RI, 2003.

\bibitem{ill_posed}
{\sc Christ, M., Colliander, J., and Tao, T.}
\newblock Ill-posedness for nonlinear {S}chrodinger and wave equations, 2003.

\bibitem{Oh_Co_Prob}
{\sc Colliander, J., and Oh, T.}
\newblock Almost sure well-posedness of the cubic nonlinear {Schr{\"o}dinger}
  equation below {{\(L^{2}(\mathbb{T})\)}}.
\newblock {\em Duke Math. J. 161}, 3 (2012), 367--414.

\bibitem{DNY2}
{\sc Deng, Y., Nahmod, A.~R., and Yue, H.}
\newblock Random tensors, propagation of randomness, and nonlinear dispersive
  equations.
\newblock {\em Invent. Math. 228}, 2 (2022), 539--686.

\bibitem{DNY1}
{\sc Deng, Y., Nahmod, A.~R., and Yue, H.}
\newblock Invariant {Gibbs} measures and global strong solutions for nonlinear
  {Schr{\"o}dinger} equations in dimension two.
\newblock {\em Ann. Math. (2) 200}, 2 (2024), 399--486.

\bibitem{Fan24}
{\sc Fan, C., and Mendelson, D.}
\newblock Construction of {{\(L^2\)}} log-log blowup solutions for the mass
  critical nonlinear {Schr{\"o}dinger} equation.
\newblock {\em J. Eur. Math. Soc. (JEMS) 26}, 5 (2024), 1795--1849.

\bibitem{GKO24}
{\sc Gubinelli, M., Koch, H., and Oh, T.}
\newblock Paracontrolled approach to the three-dimensional stochastic nonlinear
  wave equation with quadratic nonlinearity.
\newblock {\em J. Eur. Math. Soc. (JEMS) 26}, 3 (2024), 817--874.

\bibitem{HadacHerrKoch2009}
{\sc Hadac, M., Herr, S., and Koch, H.}
\newblock Well-posedness and scattering for the {KP}-{II} equation in a
  critical space.
\newblock {\em Ann. Inst. H. Poincar\'{e} Anal. Non Lin\'{e}aire 26}, 3 (2009),
  917--941.

\bibitem{HaniPausader}
{\sc Hani, Z., and Pausader, B.}
\newblock On scattering for the quintic defocusing nonlinear {S}chr\"{o}dinger
  equation on {$\Bbb R\times\Bbb T^2$}.
\newblock {\em Comm. Pure Appl. Math. 67}, 9 (2014), 1466--1542.

\bibitem{HerrTataruTz1}
{\sc Herr, S., Tataru, D., and Tzvetkov, N.}
\newblock Global well-posedness of the energy-critical nonlinear
  {S}chr\"{o}dinger equation with small initial data in {$H^1(\Bbb T^3)$}.
\newblock {\em Duke Math. J. 159}, 2 (2011), 329--349.

\bibitem{HerrTataruTz2}
{\sc Herr, S., Tataru, D., and Tzvetkov, N.}
\newblock Strichartz estimates for partially periodic solutions to
  {S}chr\"{o}dinger equations in {$4d$} and applications.
\newblock {\em J. Reine Angew. Math. 690\/} (2014), 65--78.

\bibitem{KevrekidisEtAl2015}
{\sc Kevrekidis, P.~G., Frantzeskakis, D.~J., and Carretero-Gonz\'alez, R.}
\newblock {\em The defocusing nonlinear {S}chr\"odinger equation}.
\newblock Society for Industrial and Applied Mathematics, Philadelphia, PA,
  2015.
\newblock From dark solitons to vortices and vortex rings.

\bibitem{KwakKwon}
{\sc Kwak, B., and Kwon, S.}
\newblock Critical local well-posedness of the nonlinear {S}chr\"odinger
  equation on the torus.
\newblock {\em Ann. Inst. H. Poincar\'{e} C Anal. Non Lin\'{e}aire 43}, 1
  (2026), 155--201.

\bibitem{LeeNonalg19}
{\sc Lee, G.~E.}
\newblock Local wellposedness for the critical nonlinear {Schr{\"o}dinger}
  equation on {{\( \mathbb{T}^3 \)}}.
\newblock {\em Discrete Contin. Dyn. Syst. 39}, 5 (2019), 2763--2783.

\bibitem{chara_besov}
{\sc Liu, C., Pr{\"o}mel, D.~J., and Teichmann, J.}
\newblock Characterization of nonlinear {Besov} spaces.
\newblock {\em Trans. Am. Math. Soc. 373}, 1 (2020), 529--550.

\bibitem{LuoCriticalScattering}
{\sc Luo, Y.}
\newblock Critical scattering for the nonlinear {{S}chr{\"o}dinger} equation on
  waveguide manifolds.
\newblock Preprint, {arXiv}:2506.00442 [math.{AP}] (2025), 2025.

\bibitem{NahmodStaffilani15}
{\sc Nahmod, A.~R., and Staffilani, G.}
\newblock Almost sure well-posedness for the periodic {3D} quintic nonlinear
  {Schr{\"o}dinger} equation below the energy space.
\newblock {\em J. Eur. Math. Soc. (JEMS) 17}, 7 (2015), 1687--1759.

\bibitem{NakamuraWada}
{\sc Nakamura, M., and Wada, T.}
\newblock Modified {S}trichartz estimates with an application to the critical
  nonlinear {S}chr\"odinger equation.
\newblock {\em Nonlinear Anal. 130\/} (2016), 138--156.

\bibitem{OhNonalg19}
{\sc Oh, T., Okamoto, M., and Pocovnicu, O.}
\newblock On the probabilistic well-posedness of the nonlinear
  {Schr{\"o}dinger} equations with non-algebraic nonlinearities.
\newblock {\em Discrete Contin. Dyn. Syst. 39}, 6 (2019), 3479--3520.

\bibitem{Oh_Inv}
{\sc Oh, T., Sosoe, P., and Tolomeo, L.}
\newblock Optimal integrability threshold for {Gibbs} measures associated with
  focusing {NLS} on the torus.
\newblock {\em Invent. Math. 227}, 3 (2022), 1323--1429.

\bibitem{Tzvetkov10}
{\sc Tzvetkov, N.}
\newblock Construction of a {Gibbs} measure associated to the periodic
  {Benjamin}-{Ono} equation.
\newblock {\em Probab. Theory Relat. Fields 146}, 3-4 (2010), 481--514.

\bibitem{Yue21}
{\sc Yue, H.}
\newblock Almost sure well-posedness for the cubic nonlinear {Schr{\"o}dinger}
  equation in the super-critical regime on {{\(\mathbb{T}^d\)}}, {{\(d\geq
  3\)}}.
\newblock {\em Stoch. Partial Differ. Equ., Anal. Comput. 9}, 1 (2021),
  243--294.

\bibitem{RmT1}
{\sc Zhao, Z.}
\newblock On scattering for the defocusing nonlinear {S}chr\"{o}dinger equation
  on waveguide {$\Bbb R^m\times \Bbb T$} (when {$m = 2,3$}).
\newblock {\em J. Differential Equations 275\/} (2021), 598--637.

\end{thebibliography}
\end{document}